\documentclass[11pt]{amsart}

\usepackage{amsthm,amssymb,amsfonts,amsmath,enumerate,graphicx,color}

\newcommand{\Z}{\mathbb{Z}}
\newcommand{\N}{\mathbb{N}}
\newcommand{\R}{\mathbb{R}}

\newcommand{\Id}{\mathrm{Id}}
\newcommand{\Des}{\mathrm{Des}}
\newcommand{\major}{\mathrm{major}}
\newcommand{\natmajor}{\mathrm{natmajor}}
\newcommand{\fmajor}{\mathrm{fmajor}}
\newcommand{\natfmaj}{\mathrm{natfmaj}}
\newcommand{\des}{\mathrm{des}}
\newcommand{\Neg}{\mathrm{Neg}}
\newcommand{\NNeg}{\mathrm{NNeg}}
\newcommand{\nega}{\mathrm{neg}}
\newcommand{\col}{\mathrm{col}}
\newcommand{\natmaj}{\mathrm{natmaj}}
\newcommand{\NatDes}{\mathrm{NatDes}}
\newcommand{\DNatDes}{\mathrm{DNatDes}}
\newcommand{\StDes}{\mathrm{StDes}}
\newcommand{\NDes}{\mathrm{NDes}}
\newcommand{\DNDes}{\mathrm{DNDes}}
\newcommand{\maj}{\mathrm{maj}}
\newcommand{\natdes}{\mathrm{natdes}}
\newcommand{\dnatdes}{\mathrm{dnatdes}}
\newcommand{\stdes}{\mathrm{stdes}}
\newcommand{\ndes}{\mathrm{ndes}}
\newcommand{\dndes}{\mathrm{dndes}}
\newcommand{\nmaj}{\mathrm{nmajor}}
\newcommand{\dnmaj}{\mathrm{dnmajor}}
\newcommand{\fdes}{\mathrm{fdes}}

\newcommand{\cone}{\mathrm{cone}}

\newcommand{\stat}{\mathrm{stat}}

\renewcommand{\phi}{\varphi}

\def\nat{\operatorname{nat}}
\def\ch{\operatorname{ch}}
\def\th{^{\text{th}}}

\def\e{{\boldsymbol e}}
\def\m{{\boldsymbol m}}
\def\p{{\boldsymbol p}}
\def\s{{\boldsymbol s}}

\def\v{{\boldsymbol v}}
\def\x{{\boldsymbol x}}
\def\z{{\boldsymbol z}}
\def\0{{\boldsymbol 0}}

\newcommand{\commentout}[1]{}
\def\s{\epsilon}

\newtheorem{theorem}{Theorem}[section]
\newtheorem{corollary}[theorem]{Corollary}
\newtheorem{proposition}[theorem]{Proposition}
\newtheorem{lemma}[theorem]{Lemma}

\theoremstyle{remark}
\newtheorem{example}[theorem]{Example}
\newtheorem{remark}[theorem]{Remark}
\theoremstyle{definition}
\newtheorem{definition}[theorem]{Definition}

\begin{document}

\title{Euler--Mahonian Statistics Via Polyhedral Geometry}

\author{Matthias Beck}
\address{Department of Mathematics\\
         San Francisco State University\\
         San Francisco, CA 94132}
\email{mattbeck@sfsu.edu}

\author{Benjamin Braun}
\address{Department of Mathematics\\
         University of Kentucky\\
         Lexington, KY 40506--0027}
\email{benjamin.braun@uky.edu}

\date{21 May 2013}

\thanks{The authors would like to thank Ira Gessel, Carla Savage, and the anonymous referees for their valuable comments and suggestions. 
This research was partially supported by the NSF through grants DMS-0810105, DMS-1162638 (Beck) and DMS-0758321 (Braun), and by a SQuaRE at the American Institute of Mathematics.}


\begin{abstract}
A variety of descent and major-index statistics have been defined for symmetric groups, hyperoctahedral groups, and their generalizations.
Typically associated to a pair of such statistics is an \emph{Euler--Mahonian distribution}, a bivariate polynomial encoding the statistics; such distributions often appear in rational bivariate generating-function identities.
We use techniques from polyhedral geometry to establish new multivariate identities generalizing those giving rise to many of the known Euler--Mahonian distributions.
The original bivariate identities are then specializations of these multivariate identities.
As a consequence of these new techniques we obtain bijective proofs of the equivalence of the bivariate distributions for various pairs of statistics.
\end{abstract}


\maketitle

\tableofcontents


\section{Introduction}

The symmetric group $S_n$ is the group of permutations of $[n]:=\{1,2,\ldots,n\}$, also realized as the Coxeter group $A_{n-1}$ yielding symmetries of a simplex.
For a permutation $\pi \in S_n$, the descent set is a classical object of study in combinatorics.

\begin{definition}
For $\pi\in S_n$, the \emph{descent set} of $\pi$ is
\[
\Des(\pi) := \bigl\{ j \in [n-1] : \, \pi(j) > \pi(j+1) \bigr\} \, .
\]
The \emph{descent statistic} is $\des (\pi) := \# \Des(\pi)$.
\end{definition}
The descent statistic is encoded in the \emph{Eulerian polynomial} $\sum_{ \pi \in S_n } t^{ \des (\pi) }$ and the most basic identity for Eulerian polynomials is
\begin{equation}\label{euleriangenfcteq}
  \sum_{ k \ge 0 } (k+1)^n \, t^k = \frac{ \sum_{ \pi \in S_n } t^{ \des (\pi) } }{ \left( 1 - t \right)^{ n+1 } } \, .
\end{equation}
Euler used this identity to \emph{define} Eulerian polynomials in \cite{eulereulerian} which he needed in his study of what is now called the Riemann $\zeta$-function; it is unlikely that he was aware of the connection of his polynomials to descent statistics.
For more on the interesting history regarding Eulerian polynomials, descent statistics, and algebraic geometry, see \cite{hirzebrucheulerian} and \cite[Chapter 1 Notes]{stanleyec1}.

Equation \eqref{euleriangenfcteq} has inspired a host of generalizations and extensions.
The first such extension is the following $q$-analogue of \eqref{euleriangenfcteq}, which in this form is due to Carlitz \cite{carlitzeulerian}, though with some effort one can derive it from the works of MacMahon \cite[Volume 2, Chapter IV, \S462]{macmahon}.
This extension involves a joint distribution of the descent statistic and the major index, defined as follows, together with the notation  $[m]_q := 1 + q + q^2 + \dots + q^{ m-1 }$. 
\begin{definition}
For $\pi \in S_n$, the \emph{major index} of $\pi$ is 
\[
\maj(\pi) := \sum_{ j \in \Des(\pi) } j \, .
\]
\end{definition}

\begin{theorem}[Carlitz] \label{macmahoncarlitz}
\[
  \sum_{ k \ge 0 } [k+1]_q^n \, t^k
  = \frac{ \sum_{ \pi \in S_n } t^{ \des (\pi) } q^{ \maj (\pi) } }{ \prod_{ j=0 }^n \left( 1 - t q^j \right) } \, .
\]
\end{theorem}

Note that \eqref{euleriangenfcteq} follows from Theorem~\ref{macmahoncarlitz} by setting $q=1$. 
This identity is called the \emph{Carlitz identity}, and the numerator on the right is known as an \emph{Euler--Mahonian distribution} due to the relation with Euler's work and MacMahon's original introduction of the major index.
The search for further generalizations of this identity has focused on finding new identities of the form 
\begin{equation}\label{statformgenfn}
  \sum_{ k \ge 0 } [k+1]_q^n \, t^k
  = \frac{ \sum_{ g \in G_n } t^{ \stat_1 (g) } q^{ \stat_2 (g) } }{ \prod_{ j=0 }^n h_j(t,q) } \,
\end{equation}
for various families of groups $G_n$ and statistics $\stat_1$ and $\stat_2$ defined on elements of $G_n$, together with naturally occuring families of functions $h_j(t,q)$.
This search has been successful, also producing analogous generalizations of the identities \eqref{wreatheulerianfcteq} and \eqref{deuleriangenfcteq} discussed in the next section.
To our knowledge, there are three general approaches to proving such identities:  
\begin{itemize}
\item via combinatorial/bijective proofs in the theory of partitions and their extensions;
\item via connections between permutation statistics and the theory of Coxeter groups, including connections to invariant theory and the coinvariant algebra of a Weyl group; and
\item via the theory of symmetric/quasisymmetric functions.
\end{itemize}
For more information regarding the first two approaches, see the citations listed throughout this paper.
For examples of the symmetric/quasisymmetric function approach, see \cite{hyatt,mendesremmel,shareshianwachseulerian}.

Our goal is to provide new multivariate generalizations of these identities using polyhedral geometry and lattice-point enumeration; as a consequence, we obtain new proofs of two-variable identities in the form of \eqref{statformgenfn}.
One of the benefits of the geometric approach is that it is relatively simple, the key ingredients being the triangulation of the unit cube by the braid arrangement together with careful choices of ray generators for unimodular cones.
Another benefit is that bijective proofs of the equidistribution of various pairs of statistics are obtained as immediate corollaries.

As we discuss in Remark~\ref{hilbertseries}, our multivariate identities can be viewed as Hilbert-series identities for various finely-graded algebras, i.e., algebras equipped with an $\N^n$-grading.
A Hilbert-series approach to multivariate extensions of these identities has previously been used in \cite{ABRmultivariate} and subsequent papers, emphasizing the use of descent bases for coinvariant algebras.
Our algebras and specializations are in some sense more straightforward than the previously considered ones, because the bivariate identities arise as simple specializations of our multivariate identities, requiring minimal or no additional substitutions and algebraic manipulations.
The geometric perspective also allows us to avoid the use of straightening laws and other algebraic techniques regarding coinvariant algebras.

Our paper is structured as follows.  
In Section~\ref{singlevariablebackground}, we discuss analogues of \eqref{euleriangenfcteq} for generalizations of permutation groups.
In Section~\ref{geometry}, we discuss the results we will need from integer-point enumeration and polyhedral geometry.
In Section~\ref{asection}, we use polyhedral geometry to prove Theorem~\ref{Atheorem}; this proof serves as a model for all the proofs in the paper.
We also briefly discuss connections between our approach, the theory of $P$-partitions, and the theory of affine semigroup algebras.

Section~\ref{wreathsection} contains most of our new results in the general setting of wreath products of the form $\Z_r\wr~S_n$.
These results generalize known bivariate identities due to Bagno, Bagno--Biagioli, and Chow--Mansour, which are themselves
generalizations of type-$B$ results due to Adin--Brenti--Roichman and Chow--Gessel.
As these original type-$B$ results have been of particular interest, we state our multivariate identities in this special case in Section~\ref{bigbsection}.
Also in Section~\ref{bigbsection} is a type-$B$ extension of an identity due to Chow--Gessel, one which in our approach relies heavily on the type-$B$ Coxeter arrangement; we do not know of an obvious extension of this to the wreath product case.
We close the paper with Section~\ref{dsection}, where we prove 
new type-$D$ generating-function identities.


\section{Generalized permutation groups and descents}\label{singlevariablebackground}

We discuss in this section analogues of \eqref{euleriangenfcteq} for hyperoctahedral groups, type-$D$ Coxeter groups, and wreath products of cyclic groups with symmetric groups.


The wreath product $\Z_r \wr S_n$ of a cyclic group of order $r$ with $S_n$ consists of pairs $(\pi,\s)$ where $\pi \in S_n$ and $\s\in \{\omega^{0},\omega^{1}, \ldots, \omega^{r-1}\}^n$ for $\omega:=e^{2\pi i / r}$ a primitive $r\th$ root of unity, see \cite{JamesKerber}.
Thus, $\s$ is a sequence of powers of an $r\th$ root of unity.
Elements of these groups are often called \emph{colored}, or \emph{indexed}, permutations.

\begin{remark}
By convention, for elements of $\Z_r \wr S_n$ we define additional values of $\pi$ and $\s$ as follows: $\pi_{n+1}:=n+1$, $\s_{n+1}:=1$, $\pi_{0}:=0$, and $\s_{0}:=1$.
\end{remark}

We will find it convenient to use \emph{window notation} for elements of wreath products.
If $\s_j=\omega^{c_j}$, then we will denote $(\pi,\s)$ as the \emph{window} $[\pi(1)^{c_1} \, \pi(2)^{c_2} \, \cdots \, \pi(n)^{c_n}]$.
We use the notation $j^{c_j}$ and $(\omega^{c_j},j)$ interchangeably for elements of $\{\omega^0,\omega^1,\ldots,\omega^{r-1}\}\times [n]$.
It is sometimes convenient to refer to $\pi(j)^{c_j}$ as $\pi(j)$ with \emph{color} $c_j$.

Because we will need to use inverses for these group elements, and for the sake of clarity, we review the algebraic structure of wreath products.
The element $(\pi,\s)\in\Z_r \wr S_n$ can be identified with the permutation matrix for $\pi$ where the $1$ in position $(\pi(i),i)$ is replaced by $\s_i$.
The group operation in $\Z_r \wr S_n$ is then given by matrix multiplication where entry-by-entry multiplication of non-zero terms is given by the group operation in $\Z_r$.

We next consider the special case of the hyperoctahedral group $B_n$, i.e., the Coxeter group yielding symmetries of a $\pm 1$-cube.
The group $B_n$ arises as the wreath product $\Z_2\wr S_n$ and thus consists of \emph{signed permutations} (see, e.g., \cite{reinersignedpermutationstats}), i.e., pairs $(\pi, \s)$ where $\pi \in S_n$ and $\s \in \left\{ \pm 1 \right\}^n$.
Because of this structure, it is common to associate the elements of $B_n$ to permutations $g$ of $[-n,n]\setminus \{0\}$ satisfying $g(-i)=-g(i)$ via the following map.
To the element $(\pi,\s)\in B_n$ we assign the set permutation $g_{(\pi,\s)}$ given by
\[
g(i)=\s_i\pi(i) \, .
\]
Thus, we will interchangeably write $j^1$ and $-j$ when using window notation and in definitions.

\begin{example}\label{BnCompEx}
In $B_4=\Z_2\wr~S_4$, the composition $[4^1 \, 1 \, 2^1 \, 3^1] \circ [3 \, 1^1 \, 4^1 \, 2]$ is equal to $[2^1 \, 4 \, 3 \, 1]$, since this composition maps, for example, $2\mapsto 1^1$ via $[3 \, 1^1 \, 4^1 \, 2]$ and $1^1\mapsto 4$ via $[4^1 \, 1 \, 2^1 \, 3^1]$, yielding
\[
2\mapsto 1^1 \mapsto (4^1)^1 =4 \, .
\]
This takes the matrix multiplication form
\[
\left[
\begin{array}{cccc}
0  & 1 & 0  & 0 \\
0  & 0 & -1 & 0 \\
0  & 0 & 0  & -1 \\
-1 & 0 & 0  & 0 \\
\end{array}
 \right]
\left[
\begin{array}{cccc}
0 & -1 & 0  & 0 \\
0 & 0  & 0  & 1 \\
1 & 0  & 0  & 0 \\ 
0 & 0  & -1 & 0 \\
\end{array}
 \right]
=
\left[
\begin{array}{cccc}
0  & 0 & 0 & 1 \\
-1 & 0 & 0 & 0  \\
0  & 0 & 1 & 0 \\ 
0  & 1 & 0 & 0 \\
\end{array}
 \right]
\]
The element $[4^1 \, 1 \, 2^1 \, 3^1]^{-1}$ is given by $[2 \, 3^1 \, 4^1 \, 1^1]$, since, for example, $[4^1 \, 1 \, 2^1 \, 3^1]$ sends $3\mapsto 2^1$, requiring that the inverse send $2\mapsto 3^1$.
\end{example}

For elements of $B_n$, there are several definitions of descents in the literature; we provide three of them here.
While the first applies to $B_n$, the latter two are defined for all $\Z_r \wr S_n$.

\begin{definition}\label{BnNatDes}
For an element $(\pi,\s)\in B_n$, the \emph{naturally ordered descent set} is
\begin{equation}\label{Bdescentdef}
  \NatDes(\pi, \s) := \bigl\{ j \in \left\{ 0, 1, \dots, n-1 \right\} : \, \s_j \pi(j) > \s_{ j+1 } \pi(j+1) \bigr\} \, ,
\end{equation}
with the convention $\s_0 \pi(0) = 0$.
The \emph{natural descent statistic} for $B_n$ is $\natdes(\pi,\s):=\#\NatDes(\pi,\s)$.
\end{definition}
The reason for calling this the \emph{naturally} ordered descent set is that it uses the natural order $-n<-n+1<\cdots<-1<1<2<\cdots<n$ on the integers.
For example, the permutation 
\[
[3^1 \, 2 \, 1^1 \, 4^1 ]
\]
in $B_4$ has descents in the zeroth, second, and third positions.

In \cite{Steingrimsson}, Steingr{\'{\i}}msson defined the following descent set for elements of $\Z_r \wr S_n$.

\begin{definition}\label{BnStDes}
Totally order the elements of $\{\omega^0,\omega^1,\ldots,\omega^{r-1}\}\times [n]$ by $j^{c_j}<k^{c_k}$ if $c_j<c_k$ or if both $c_j=c_k$ and $j<k$ hold.
For an element $(\pi, \s) = [\pi(1)^{c_1} \, \pi(2)^{c_2} \, \cdots \, \pi(n)^{c_n}]$ in $\Z_r \wr S_n$, \emph{Steingr{\'{\i}}msson's descent set} is 
\begin{equation}\label{Steinwreathdescentdef}
  \StDes(\pi, \s) := \bigl\{ j \in \left\{ 1, \dots, n \right\} : \, \pi(j)^{c_j}>\pi(j+1)^{c_{j+1}} \bigr\} .
\end{equation}
\emph{Steingr{\'{\i}}msson's descent statistic} is $\stdes (\pi,\s) := \#\StDes(\pi,\s)$.
\end{definition}

As an example, observe that with Steingr{\'{\i}}msson's ordering we have $\{\omega^0,\omega^1,\omega^{2}\}\times [3]$ ordered as
\[
1^0<2^0<3^0<1^1<2^1<3^1<1^2<2^2<3^2 \, ,
\]
and the permutation
\[
[2^2 \, 3^2 \, 1^1 ]
\]
has descents in positions $2$ and $3$.

Finally, we define the following closely-related descent set.
This definition differs from Steingr{\'{\i}}msson's both in the role played by the order of the roots of unity and in the indices where descents may occur.

\begin{definition}\label{BnDes}
Totally order the elements of $\{\omega^{r-1},\omega^{r-2},\ldots,\omega^{0}\}\times [n]$ by $j^{c_j}<k^{c_k}$ if $c_j>c_k$ or if both $c_j=c_k$ and $j<k$ hold.
For an element $(\pi, \s) = [\pi(1)^{c_1} \, \pi(2)^{c_2} \, \cdots \, \pi(n)^{c_n}]$ in $\Z_r \wr S_n$, the \emph{descent set} is 
\begin{equation}\label{wreathdescentdef}
  \Des(\pi, \s) := \bigl\{ j \in \left\{ 0, \dots, n-1 \right\} : \, \pi(j)^{c_j}>\pi(j+1)^{c_{j+1}} \bigr\} .
\end{equation}
The \emph{descent statistic} is $\des (\pi,\s) := \#\Des(\pi,\s)$.
\end{definition}

As an example, observe that with this order we have $\{\omega^0,\omega^1,\omega^{2}\}\times [3]$ ordered as
\[
1^2<2^2<3^2<1^1<2^1<3^1<1^0<2^0<3^0  \, ,
\]
and the permutation
\[
[3^2 \, 2^0 \, 1^1 ]
\]
has descents in positions $0$ and $2$.

The Eulerian polynomials for wreath products are $\sum_{(\pi,\s)\in \Z_r \wr S_n}t^{\des(\pi,\s)}$, where one may use either of the two wreath product descent definitions or, in the case $r=2$, the natural descent statistic.
The resulting analogue of \eqref{euleriangenfcteq} is
\begin{equation}\label{wreatheulerianfcteq}
  \sum_{ k \ge 0 } (rk+1)^n \, t^k
  = \frac{ \sum_{(\pi,\s)\in \Z_r \wr S_n}t^{\des(\pi,\s)} }{(1-t)^{n+1} } \, .
\end{equation}
This identity appears to have been found by various authors for different descent statistics; for more details, see \cite{brentieulerian,Steingrimsson}.


\section{A geometric perspective}\label{geometry}


\subsection{Simplices and cones}

The forms of equations \eqref{euleriangenfcteq}, \eqref{wreatheulerianfcteq}, and \eqref{deuleriangenfcteq} suggest that one should look at them geometrically as stemming from lattice-point enumeration of the cube $[0,r]^n$ as it is partitioned in various ways; for example, \eqref{euleriangenfcteq} suggests we consider $[0,1]^n$ partitioned by the braid arrangement consisting of the hyperplanes $x_j = x_k$ for $1 \le j < k \le n$. 
As a result of such partitions, we will encounter certain simplices throughout this work, all of which are (after a suitable change of variables) of the form
\[
  \Delta_I := \left\{ \x \in \R^n :
     \begin{array}{ll}
       0 \le x_n \le x_{n-1} \le \dots \le x_1 \le 1, \\
       x_{j+1} < x_{ j } \text{ if } j \in I
     \end{array}
  \right\} ,
\]
where $I \subseteq [n]$ is some index set, and we use the convention $x_n > 0$ if $n \in I$.

\begin{remark}
The definition of $\Delta_I$ we have given is technically that of a simplex with some of its facets removed.
Throughout this work, we will be decomposing cubes into disjoint unions of such objects; the removal of facets will be needed to ensure that our decompositions are disjoint.
In the following, to simplify nomenclature, we will freely refer to these partially open objects as \emph{simplices}.
Further, for a polyhedron $P$ with some facets removed, we will use the terms \emph{faces} and \emph{vertices} of $P$ to refer to the faces and vertices of the closure of $P$.
\end{remark}

The vertices of $\Delta_I$ are $\0, \, \e_1 + \dots + \e_n, \, \e_1 + \dots + \e_{n-1}, \dots, \, \e_{1 } + \e_2, \, \e_1$, where $\e_j$ is the $j$'th unit vector in $\R^n$.
Note that $\Delta_I$ is \emph{unimodular}, i.e., the $n$ edge directions at any vertex of $\Delta_I$ generate $\Z^n$.
The \emph{cone over} $\Delta_I$ is the nonnegative span of $\left\{ (1,\x) \in \R^{ n+1 } : \, \x \in \Delta_I \right\}$, where we encode the ``new" dimension by the variable $x_0$, i.e.,
\[
  \cone \left( \Delta_I \right) := \R_{ \ge 0 } \, \e_0 + \sum_{ j \in I } \R_{ >0 } \left( \e_0 + \e_{ 1 } + \e_{ 2 } + \dots + \e_j \right) + \sum_{ j \notin I } \R_{ \ge 0 } \left( \e_0 + \e_{1 } + \e_{ 2 } + \dots + \e_j \right) ,
\]
where the complement of $I$ is taken in $[n]$.


\subsection{Generating functions for cones}

Let
\[
  \sigma_C (z_0, z_1, \dots, z_n) := \sum_{ \m \in C \cap \Z^{ n+1 } } \z^\m
\]
be the multivariate (``full") generating function encoding the integer lattice points in a subset $C\subset \R^{n+1}$, where we have used the shorthand $\z^\m := z_0^{ m_0 } z_1^{ m_1 } \cdots z_n^{ m_n }$.
A standard geometric-series argument (see, e.g., \cite[Theorem 3.5]{ccd}), together with the unimodularity of $\cone \left( \Delta_I \right)$, gives the following.

\begin{lemma}\label{conelemma}
Let $\Delta_I$ be as above. Then \
\[
  \sigma_{ \cone \left( \Delta_I \right) } (z_0, z_1, \dots, z_n) = \frac{ \prod_{ j \in I } z_0 \, z_{ 1 } z_{ 2 } \cdots z_j }{ \prod_{ j=0 }^n \left( 1 - z_0 \, z_{ 1 } z_{ 2 } \cdots z_j \right) } \, .
\]
\end{lemma}

We will not always use the above natural way to write the generating function of a unimodular cone, in which case we will apply the following more general lemma.  
The proof is a straightforward extension of \cite[Theorem 3.5 and Corollary 3.6]{ccd} and \cite[Corollary 4.6.8 and its Note]{stanleyec1}; only the latter reference discusses the relationship between determinants and monomials stated here.

\begin{lemma}\label{genconelemma}
Let $C = \sum_{ j=0 }^k \R_{ \ge 0 } \v_j + \sum_{ j=k+1 }^n \R_{ > 0 } \v_j$ be a half-open simplicial cone in $\R^{ n+1 }$ with linearly independent generators $\v_0, \v_1, \dots, \v_n \in \Z^{ n+1 }$. Then
\[
  \sigma_C (z_0, z_1, \dots, z_n) = \frac{ \sigma_{ \Pi_C } (z_0, z_1, \dots, z_n) }{ \prod_{ j=0 }^n \left( 1 - \z^{ \v_j } \right) }
\]
where $\Pi_C := \sum_{ j=0 }^k [0,1) \v_j + \sum_{ j=k+1 }^n (0,1] \v_j$.
Furthermore, the number of integer points in $\Pi_C$ (and thus the number of monomials in $\sigma_{ \Pi_C } (z_0, z_1, \dots, z_n)$) is given by the determinant of the matrix with column vectors $\v_0, \v_1, \dots, \v_n$.
\end{lemma}
We refer to the set $\Pi_C$ arising in the lemma as the \emph{fundamental parallelepiped} of $C$; note that it depends on the choice of generators of~$C$.


\subsection{Unimodular cones with scaled ray generators}\label{sec:unim}

Throughout this work we will frequently need to compute $\sigma_{ \Pi_C } (z_0, z_1, \dots, z_n)$ for a unimodular cone $C$ of the form given in Lemma~\ref{genconelemma}, where the generators chosen for the cone are not the minimal length ray generators.
Using the notation of Lemma~\ref{genconelemma}, let $\v_0,\ldots,\v_n$ denote the minimal ray generators for a unimodular cone $C$, so that
\[
\sigma_{\overline{C}} (z_0, z_1, \dots, z_n) = \frac{ 1 }{ \prod_{ j=0 }^n \left( 1 - \z^{ \v_j } \right) } \, ,
\]
where $\overline{C}$ denotes the topological closure of $C$.
If we use instead the ray generators $c_0\v_0,c_1\v_1,\ldots,c_n\v_n$ for some positive integer scaling factors $c_0,c_1,\ldots,c_n$, we will desire in this paper to obtain the integer points in 
\[
\Pi_{C}=\sum_{ j=0 }^k [0,1) c_j\v_j~+~\sum_{ j=k+1 }^n (0,1] c_j\v_j
\]
from the integer points in 
\[
\Pi_{\overline{C}}=\sum_{ j=0 }^n [0,1) c_j\v_j \, .
\]
Since $C$ is unimodular with ray generators given by the $\v_j$'s, the integer points in $\Pi_{\overline{C}}$ are those integer points of the form
\[
\p=\sum_{j=0}^n \alpha_j\v_j
\]
where $0\leq \alpha_j< c_j$ is an integer.
Thus, there are $\prod_jc_j$ integer points contained in $\Pi_{\overline{C}}$.
Observe that $\p$ lies on the facet of $C$ opposite $\v_j$ if and only if $\alpha_j=0$.
Thus, each integer point $\p$ in the set $\Pi_{\overline{C}}$ with $\alpha_j=0$ for some indices $j\geq k+1$ does not lie in the set $\Pi_C$.
Similarly, each integer point in $\Pi_C$ of the form $\p=\sum_{j=0}^n \alpha_j\v_j$, such that $\alpha_j=c_j$ for some $j\geq k+1$, is not in $\Pi_{\overline{C}}$.
If we fix an index set $J\subseteq \{k+1,k+2,\ldots,n\}$, there is a bijective correspondence between the points
\[
\p_1=\sum_{j=0}^n \alpha_j\v_j\in \Pi_{\overline{C}}
\]
where $\alpha_j=0$ if $j\in J$ and the points
\[
\p_2=\sum_{j=0}^n \beta_j\v_j\in \Pi_C
\]
where $\beta_j=c_j$ if $j\in J$.
This bijection is obtained by identifying two such points when $\alpha_j=\beta_j$ for all $j\notin J$.

In the following, we will use one of the following two techniques to obtain the set $\Z^n\cap \Pi_C$ from $\Z^n\cap \Pi_{\overline{C}}$.
\begin{itemize}
\item \emph{Shifting integer points off the boundary:} Each integer point $\p\in\Pi_{\overline{C}}\setminus \Pi_C$ is of the form $\p=\sum_{j=0}^n \alpha_j\v_j$ where $0\leq \alpha_j< c_j$ and $\alpha_j=0$ for all $j\in J_{\p} \subset \{k+1,k+2,\ldots,n\}$ for some index set $J_{\p}$.
By shifting each such $\p$ by $\sum_{j\in J_{\p}}c_jv_j$, we obtain
\[
\Z^n\cap \Pi_C = \left[\Z^n\cap\Pi_{\overline{C}}\cap \Pi_C\right]\cup \left\{\p+\sum_{j\in J_{\p}}c_j\v_j: \p\in\Pi_{\overline{C}}\setminus \Pi_C \right\} .
\]

\item \emph{Shifting the entire parallelepiped:} Alternatively, we may observe that $\Pi_{\overline{C}}$ is a parallelepiped with half of its facets removed, where no two opposite pairs of facets are simultaneously removed.
Similarly, $\Pi_C$ is a parallelepiped of the same type, but with a different selection of included facets.
Thus, it is immediate that
\[
\Z^n\cap \Pi_C = \left[\Z^n\cap \Pi_{\overline{C}}\right]+\left(\sum_{\substack{i :\text{ the facet opposite}\\\v_i \text{ is removed in }C}} \v_i \right) .
\]

\end{itemize}

\begin{example}
Let $C=\{\x\in \R^3:0\leq x_3 <x_2 <x_1\}$.
Thus, $\overline{C}\subset \R^3$ is generated by $\v_1=(1,0,0)$, $\v_2=(1,1,0)$, and $\v_3=(1,1,1)$.
Using the ray generators $2\v_1$, $2\v_2$, and $3\v_3$ for $\overline{C}$, there are three integer points in $\Pi_C\cap \Pi_{\overline{C}}$ given by
\[
(2,1,0), (3,2,1), (4,3,2) \, .
\]
There are nine integer points in $\Pi_{\overline{C}}\setminus \Pi_C$, namely
\[
(0,0,0), (1,1,1), (2,2,2), (1,1,0), (1,0,0), (2,2,1), (3,3,2), (3,2,2), (2,1,1) \, .
\]
Similarly, there are nine integer points in $\Pi_C\setminus \Pi_{\overline{C}}$, namely
\[
(4,2,0), (5,3,1), (6,4,2), (3,1,0), (3,2,0), (4,3,1), (5,3,2), (5,4,2), (4,2,1) \, .
\]
It is straightforward to check that both of the shifting methods described above produce the integer points in $\Pi_C$ from the integer points in $\Pi_{\overline{C}}$.
Shifting off the boundary adds one or both of $(2,0,0)$ and $(2,2,0)$ to the points of $\Pi_{\overline{C}}\setminus \Pi_C$,
while shifting the parallelepiped adds $(2,1,0)$ to all the points of $\Pi_{\overline{C}}$.
\end{example}


\section{Type $A$}\label{asection}

We begin with a multivariate identity that specializes to Theorem~\ref{macmahoncarlitz}.
The proof of this identity, though simple, demonstrates the approach used in this paper.

\begin{theorem}\label{Atheorem}
\[
  \sum_{ k \ge 0 } \prod_{ j=1 }^n [k+1]_{ z_j } \, z_0^k
  \ = \ \sum_{ \pi \in S_n } \frac{ \prod_{ j \in \Des(\pi) } z_0z_{ \pi(1) } z_{ \pi(2) } \cdots z_{ \pi(j) }  }{ \prod_{ j=0 }^n \left( 1 - z_0 \, z_{ \pi(1) } z_{ \pi(2) } \cdots z_{ \pi(j) }  \right) } \, .
\]
\end{theorem}

\begin{proof}
Triangulate the $n$-cube $[0,1]^n$ into the disjoint union of simplices
\[
  \Delta_\pi := \left\{ \x \in \R^n :
     \begin{array}{ll}
       0 \le x_{ \pi(n) } \le x_{ \pi(n-1) } \le \dots \le x_{ \pi(1) } \le 1, \\
       x_{ \pi(j+1) } < x_{ \pi(j) } \text{ if } j \in \Des(\pi)
     \end{array}
  \right\} 
\]
(one for each $\pi \in S_n$).
Lemma 4.5.1 of \cite{stanleyec1} implies that the strict inequalities determined by the descent set of $\pi$ make this triangulation \emph{disjoint}.
For example, if $\x=(x_1,\ldots,x_9)=(.2,.1,.2,.3,.1,.1,.3,.3,.2) \in [0,1]^9$, then $\x\in\Delta_\pi$ where $\pi=[4,7,8,1,3,9,2,5,6]$, since $x_6=x_5=x_2< x_9=x_3=x_1<  x_8=x_7=x_4$.
By Lemma \ref{conelemma},
\[
  \sigma_{ \cone(\Delta_\pi) } (z_0, z_1, \dots, z_n) = \frac{ \prod_{ j \in \Des(\pi) } z_0 \, z_{ \pi(1) } z_{ \pi(2) } \cdots z_{ \pi(j) }  }{ \prod_{ j=0 }^n \left( 1 - z_0 \, z_{ \pi(1) } z_{ \pi(2) } \cdots z_{ \pi(j) }  \right) } \, .
\]
On the other hand,
\[
  \sigma_{ \cone([0,1]^n) } (z_0, z_1, \dots, z_n) = \sum_{ k \ge 0 } \prod_{ j=1 }^n \left( 1 + z_j + z_j^2 + \dots + z_j^k \right) z_0^k \, ,
\]
and the disjoint triangulation gives
\[
  \sigma_{ \cone([0,1]^n) } (z_0, z_1, \dots, z_n) = \sum_{ \pi \in S_n } \sigma_{ \cone(\Delta_\pi) } (z_0, z_1, \dots, z_n) \, . \qedhere
\]
\end{proof}

\begin{proof}[Proof of Theorem \ref{macmahoncarlitz}]
Setting $t := z_0$ and $q := z_1 = z_2 = \dots = z_n$ in Theorem \ref{Atheorem} gives
\[
  \sum_{ k \ge 0 } [k+1]_q^n \, t^k
  = \frac{ \sum_{ \pi \in S_n } \prod_{ j \in \Des(\pi) } t q^{ j } }{ \prod_{ j=0 }^n \left( 1 - t q^{ j } \right) }
  = \frac{ \sum_{ \pi \in S_n } t^{ \des (\pi) } q^{ \maj (\pi) } }{ \prod_{ j=0 }^n \left( 1 - t q^j \right) } \, . \qedhere
\]
\end{proof}

\begin{remark}
Our approach is related to the theory of $P$-partitions \cite{stanleythesis,stanleyec1}.
For a given finite poset $P$, one can associate a cone of $P$-partitions.
The standard approach to studying $P$-partitions, going back to Stanley's pioneering work referenced above, is to recognize that each $P$-partition cone is a union of closed chambers of the type-$A$ braid arrangement.
Thus, each $P$-partition cone admits a unimodular triangulation, and these unimodular subcones are indexed by linear extensions of $P$.

Our approach is based almost entirely on the triangulation of $[0,1]^n$ induced by the type-$A$ braid arrangement; the relationship with $P$-partitions is then that $[0,1]^n$ is a truncation of the $P$-partition cone in the case where $P$ is an antichain of size~$n$.
That the linear extensions of such an antichain are easily put into bijection with the elements of $S_n$ gives our connection to symmetric groups and the braid arrangement.
Mirroring these similarities, our Theorem~\ref{Atheorem} resembles \cite[Theorem~7.1]{stanleythesis}.

Where our techniques diverge from being a minor variant of $P$-partition theory is that throughout this work, when we encounter a unimodular triangulation of $\cone([0,1]^n)$, we often choose non-unimodular generators for the unimodular cones in our triangulation.
Also, several of our generating-function identities require studying non-unimodular triangulations of $\cone([0,r]^n)$ for $r\geq 2$.
To our knowledge, this approach has not been used in the study of $P$-partitions.
\end{remark}

\begin{remark}\label{hilbertseries}
There is also a connection between our generating functions and the theory of affine semigroup algebras.
The generating function in Theorem~\ref{Atheorem} is the finely-graded Hilbert series for the affine semigroup algebra formed from the semigroup of integer points in $\cone([0,1]^n)$, as discussed in \cite{HibiBook,MillerSturmfels,StanleyGreenBook}.
Through much of the recent literature on Euler--Mahonian distributions referenced in this paper, Hilbert-series interpretations for these bivariate identities have been sought.
All of our identities provide such interpretations, as they arise from the finely-graded Hilbert series of affine semigroup algebras.

Further, the study of semigroup algebras arising from polyhedral cones has been an area of intense study for combinatorial commutative algebraists over the past several decades.
The most important general result regarding Hilbert series for such cones is Hochster's theorem, which states that normal affine semigroup algebras are Cohen--Macaulay \cite{Hochster}.
The Cohen--Macaulay property forces serious constraints on single-variable specializations of the associated finely-graded Hilbert series for the algebra; these constraints apply to univariate specializations of our identities.
\end{remark}


\section{Wreath products}\label{wreathsection}


In this section, we prove three new multivariate generating function identities connected with pairs of statistics on wreath products of the form $\Z_r\wr~S_n$.
Our proofs of these identities lead to a bijective proof of the joint equidistribution of the ``negative'' and ``flag'' statistics.


\subsection{Identities involving $(k+1)^n$}

We begin by recalling the definition of the negative statistics and flag statistics on $\Z_r \wr S_n$, as introduced in \cite{bagno, bagnobiagioli}.
These are generalizations of the type-$B$ negative and flag statistics introduced by Adin--Brenti--Roichman, which we discuss in Section~\ref{bigbsection}.
Our interest in these statistics comes from the role they play in the following two identities.

\begin{theorem}[Bagno, \cite{bagno}]\label{wreathnegBthm}
\[
  \sum_{ k \ge 0 } [k+1]_q^n \, t^k
  = \frac{ \sum_{(\pi,\s) \in \Z_r\wr S_n } t^{ \ndes(\pi,\s) } q^{ \nmaj(\pi,\s) } }{ (1-t)\prod_{j=1}^n (1-t^rq^{rj}) } \, .
\]
\end{theorem}

\begin{theorem}[Bagno--Biagioli, \cite{bagnobiagioli}]\label{wreathflagBthm}
\[
  \sum_{ k \ge 0 } [k+1]_q^n \, t^k
  = \frac{ \sum_{(\pi,\s) \in \Z_r\wr S_n } t^{ \fdes(\pi,\s) } q^{ \fmajor(\pi,\s) } }{ (1-t)\prod_{j=1}^n (1-t^rq^{rj}) } \, .
\]
\end{theorem}

\begin{remark}
Bagno and Biagioli also prove in \cite{bagnobiagioli} a multivariate theorem of this type for a family of normal subgroups of $\Z_r \wr S_n$.
Their techniques involve studying colored-descent representations of these subgroups, which are representations of the groups on the associated coinvariant algebra.
\end{remark}

Throughout this subsection, we use the total order from Definition~\ref{BnDes} on the elements of $\{\omega^{r-1},\omega^{r-2},\ldots,\omega^{0}\}\times [n]$, i.e., $j^{c_j}<k^{c_k}$ if $c_j>c_k$ or if both $c_j=c_k$ and $j<k$ hold.

\begin{definition}\label{typeAwreathstats}
For an element $(\pi,\s)\in \Z_r \wr S_n$, we define the \emph{negative set} of $(\pi,\s)$ to be
\[
\Neg(\pi,\s):= \{ i\in [n]: \s_i \neq \omega^{0}=1 \} \, ,
\]
and we define $\nega(\pi,\s):=\#\Neg(\pi,\s)$.
Writing $\s_j=\omega^{c_j}$, we define the \emph{color sum statistic} to be
\[
\col(\pi,\s):=\sum_{i\in [n]} c_i \, .
\]
The \emph{type-$A$ descent set} is defined to be
\[
\Des_A(\pi,\s):= \{i\in [n-1] : \pi_i^{c_i} > \pi_{i+1}^{c_{i+1}} \} \, 
\]
and the \emph{type-$A$ descent statistic} is
\[
\des_A(\pi,\s):=\#\Des_A(\pi,\s) \, .
\]
The \emph{type-$A$ major index} is
\[
\major_A(\pi,\s)=\sum_{j\in \Des_A(\pi,\s)}j \, .
\]
\end{definition}

\begin{example} 
Let $(\pi,\s)=[1^3 \, 4^0 \, 2^1 \, 3^0 \, 6^2 \, 5^1 ]\in \Z_4\wr S_6$.
Then
\[
\Neg(\pi,\s)=\{1,3,5,6\} 
\]
and $\col(\pi,\s)=3+1+2+1=7$.
Further,
\[
\Des_A(\pi,\s)=\{2,4\}
\]
and thus $\des_A(\pi,\s)=2$ and $\major_A(\pi,\s)=6$.
\end{example}

We next define \emph{negative} statistics for wreath products, following \cite{bagno}.
Recall first that a \emph{multiset} of elements of $[n]$ is a subset $S\subseteq[n]$ together with a function $\nu:S\rightarrow \Z_{\geq 1}$, where we call $\nu(i)$ is the \emph{multiplicity} of $i$ in $S$.
Instead of specifying $\nu$ for a multiset, we typically write a multiset as a set of elements with repetition, e.g. $M=\{1,1,1,2,4,4,4,4,7,7\}$ represents the multiset with $S=\{1,2,4,7\}$ where $\nu(1)=3$, $\nu(2)=1$, $\nu(4)=4$, and $\nu(7)=2$.
The cardinality of a multiset is the sum of the multiplicities of the elements of the underlying set.
To form a union of multisets, we take the union of the underlying sets and sum the multiplicities of the elements.
When forming a sum (or product) indexed by the elements of a multiset $(S,\nu)$, we include $\nu(i)$ summands (or factors) for each $i\in S$.
For example, with our previous example $M$, we have $\sum_{i\in M}2^i=3\cdot 2^1 + 1\cdot 2^2 + 4\cdot 2^4 + 2\cdot 2^7$.

\begin{definition} 
For an element $(\pi,\s)$ in $\Z_r\wr S_n$, we define the \emph{negative inverse multiset} as
\[
\NNeg(\pi,\s) :=\{\underbrace{i,i,\ldots,i}_{c_i \text{ times}}:i\in[n]\} \, .
\]
We define the \emph{negative descent multiset} as
\[
\NDes(\pi,\s):= \Des_A(\pi,\s)\cup \NNeg((\pi,\s)^{-1}) \, .
\]
The \emph{negative descent statistic} is 
\[
\ndes(\pi,\s):=\#\NDes(\pi,\s) \, .
\]
The \emph{negative major index} is
\[
\nmaj(\pi,\s):=\sum_{i\in \NDes(\pi,\s)}i \, .
\]
\end{definition}

\noindent
Observe that $\NNeg((\pi,\s)^{-1})$ contains exactly $(r-c_{\pi^{-1}(i)})\mod r$ copies of each $i\in [n]$.

\begin{example} 
Let $(\pi,\s)=[1^3 \, 4^0 \, 2^1 \, 3^0 \, 6^2 \, 5^1 ]\in \Z_4\wr S_6$.
Then $(\pi,\s)^{-1}=[1^1 \, 3^3 \, 4^0 \, 2^0 \, 6^3 \, 5^2 ]$, and hence $\NNeg((\pi,\s)^{-1})=\{1,2,2,2,5,5,5,6,6\}$.  
There are $r-c_{\pi^{-1}(5)}=4-c_6=4-1=3$ copies of $5$ contained in this set, and there are
$r-c_{\pi^{-1}(3)}=4-c_4=4-0\equiv 0 \bmod 4$ copies of $3$.
Further,
\[
\NDes(\pi,\s)=\{2,4\}\cup\{1,2,2,2,5,5,5,6,6\}=\{1,2,2,2,2,4,5,5,5,6,6\}
\]
and thus $\ndes(\pi,\s)=11$ and $\nmaj(\pi,\s)=40$.
\end{example}

There are also flag statistics for wreath products, due to Bagno and Biagioli \cite{bagnobiagioli}.

\begin{definition}[Bagno--Biagioli]\label{wreathflagdef}
For an element $(\pi,\s)$ in $\Z_r\wr S_n$, we define the \emph{flag descent statistic} as
\[
\fdes(\pi,\s):= r\cdot \des_A(\pi,\s) + c_1 \, ,
\]
where as usual $\s_1=\omega^{c_1}$.
The \emph{flag major index} is
\[
\fmajor(\pi,\s):= r\cdot \major_A(\pi,\s) + \col(\pi,\s) \, .
\]
\end{definition}

\begin{example} 
Let $(\pi,\s)=[1^3 \, 4^0 \, 2^1 \, 3^0 \, 6^2 \, 5^1 ]\in \Z_4\wr S_6$.
Then $\fdes(\pi,\s)=4\cdot 2 + 3=11$ and $\fmajor(\pi,\s)=4\cdot 6+7=31$.
\end{example}

For the statements of our multivariate generalizations of Theorems~\ref{wreathnegBthm} and~\ref{wreathflagBthm}, we will need two more definitions.

\begin{definition}\label{Irn}
Define the subset of \emph{increasing elements of $\Z_r \wr S_n$}, denoted $I_{r,n}$, to be those elements satisfying $\des_A(\rho,\s)=0$, i.e., $I_{r,n}$ contains all permutations $(\rho,\s)$ such that $\rho(j)^{c_j}<\rho(j+1)^{c_{j+1}}$ for all $j\in [n-1]$. 
\end{definition}
It is straightforward that every element of $\Z_r \wr S_n$ can be represented uniquely as 
\[
(\rho,\s)\circ(\pi,(1,1,\cdots, 1))
\]
 for some $\pi\in S_n$ and $(\rho,\s)\in I_{r,n}$, since applying $(\pi,(1,1,\cdots, 1))$ on the right permutes the entries of the window notation for $(\rho,\s)$, and the window for $(\rho,\s)$ yields the unique increasing list of these entries.
For example, in $B_6=\Z_2\wr~S_6$, 
\[
[4^1 \, 1^1 \, 5 \, 3^1 \, 6 \, 2] = [1^1 \, 3^1 \, 4^1 \, 2 \, 5 \, 6][3\, 1 \, 5 \, 2 \, 6 \, 4] \, .
\]
Thus, 
\[
\Z_r \wr S_n=\bigcup_{\pi\in S_n} I_{r,n}\pi \, ,
\]
where we write $\pi$ for $(\pi,(1,1,\cdots, 1))$ to simplify notation.

\begin{proposition}\label{prop:IncNeg}
For $(\rho,\s)\in I_{r,n}$ and $\pi\in S_n$,
\[
\NNeg([(\rho,\s)\pi]^{-1})=\NNeg((\rho,\s)^{-1}) \, .
\]
Further, each permutation $(\rho,\s)\in I_{r,n}$ is uniquely determined by $\NNeg((\rho,\s)^{-1})$.
\end{proposition}

\begin{proof}
If $(\tau,\s)=[\tau_1^{c_1} \, \cdots \, \tau_n^{c_n} ]$ is any element of $\Z_r\wr S_n$, then 
\[
\NNeg((\tau,\s)^{-1})=\{\underbrace{\tau_i,\ldots,\tau_i}_{\substack{(r-c_i) \bmod r\\ \text{ times }}}:c_i\neq 0\} \, .
\]
Since the window for $(\rho,\s)\pi$ consists of a permutation of the window elements for $(\rho,\s)$, and each $\rho_i^{c_i}$ is permuted as a unit by $\pi$ from the window of $(\rho,\s)$ to the window for $(\rho,\s)\pi$, it follows that the labels $\rho_i^{c_i}$ in the window are identical for both these permutations.
The first claim follows.

To verify the uniqueness statement, it is enough to observe that $\NNeg((\rho,\s)^{-1})$ determines the exponent on each $i\in [n]$ in the window notation for $(\rho,\s)$.
Since being an element in $I_{r,n}$ ensures that the entries of the window for $(\rho,\s)$ are in increasing order, this determines the permutation.
\end{proof}

\begin{definition}\label{wreathdefsignchange}
For an element $\s=(\omega^{c_j})_{j=1}^n\in~\{1,\omega^1,\omega^2,\ldots,\omega^{r-1}\}^n$ with $\s_{n+1}:=1=\omega^0$, for $j\in [n]$ define
\[
a_j^\s:= (c_j-c_{j+1})\mod r \, ,
\]
which we call the \emph{$j$th color change} for $\s$.
Define $\ch(\s):=\sum_ja_j^\s$ to be the \emph{total color change} in $\s$.
\end{definition}

\begin{example} 
Let $(\pi,\s)=[1^3 \, 4^0 \, 2^1 \, 3^0 \, 6^2 \, 5^1 ]\in \Z_4\wr S_6$, so that $\s=(\omega^3, \omega^0, \omega^1,\omega^0,\omega^2,\omega^1)$.
Then
\[
a^\s=(3,3,1,2,1,1)
\]
and
$\ch(\s)=3+3+1+2+1+1=11$.
\end{example}


\subsection{Multivariate identities}

Our multivariate extension of Theorem~\ref{wreathnegBthm} is the following.

\begin{theorem}\label{wreathcubeBthm}
\begin{align*}
  \sum_{ k \ge 0 } \prod_{ j=1 }^n [k+1]_{ z_j } \, z_0^k \ = \ 
\sum_{\pi\in S_n} \sum_{(\rho,\s)\in I_{r,n}}  &  \frac{\displaystyle \prod_{j\in \Des(\pi)} z_0z_{\pi(1)}z_{\pi(2)}\cdots z_{\pi(j)} \prod_{j\in \NNeg((\rho,\s)^{-1})} z_0z_{\pi(1)}z_{\pi(2)}\cdots z_{\pi(j)} }{\displaystyle (1-z_0)\prod_{j=1}^n \left(1-z_0^r z_{\pi(1)}^r\cdots z_{\pi(j)}^r \right) } \, .
\end{align*}
\end{theorem}

\begin{proof}
We begin with the triangulation of $\cone ([0,1]^n)$ into the set of cones 
$
\left\{ \cone(\Delta_\pi) : \, \pi \in S_n \right\}
$
found in the proof of Theorem~\ref{Atheorem}.
While $\cone(\Delta_\pi)$ is unimodular for each $\pi$, for this proof we use the non-unimodular ray generators 
\[
\e_0, r(\e_0+\e_{\pi(1)}), r(\e_0+\e_{\pi(1)}+\e_{\pi(2)}), \ldots, r(\e_0+\e_{\pi(1)}+\cdots + \e_{\pi(n)}) \, ,
\]
together with the technique of shifting the entire fundamental parallelepiped described in Section~\ref{sec:unim}.
There are $r^n$ integer points in the fundamental parallelepiped for $\overline{\cone(\Delta_\pi)}$ using these ray generators.
Thus, every integer point $\p$ in the fundamental parallelepiped for $\cone(\Delta_\pi)$ can be uniquely expressed as
\[
\p=\sum_{j\in \Des(\pi)} (\e_0+\e_{\pi(1)}+\cdots + \e_{\pi(j)}) + \sum_{j=1}^n \alpha_j(\e_0+\e_{\pi(1)}+\cdots + \e_{\pi(j)}) \, 
\]
with $\alpha_j\in \{0,1,\ldots,r-1\}$.

Associate to the point $\p$ the element $(\rho,\s)\pi \in \Z_r \wr S_n$, where $\alpha_j=k$ if and only if $j$ has multiplicity $k$ in $\NNeg((\rho,\s)^{-1})=\NNeg([(\rho,\s)\pi]^{-1})$.
Thus, for example, let $r=4$ and $n=6$, and consider $\pi=[1\,6\, 3\, 5\, 2\, 4]$ and $\alpha_1=1$, $\alpha_2=3$, $\alpha_3=\alpha_4=0$, $\alpha_5=3$, and $\alpha_6=2$.
The element of $\Z_4\wr S_6$ associated to this point is $[1^3 \, 4^0 \, 2^1 \, 3^0 \, 6^2 \, 5^1]$, since it is contained in $I_{4,6}\pi$ and has the $\NNeg$ set of its inverse equal to $\{1,2,2,2,5,5,5,6,6\}$.

This correspondence creates a bijection between the elements of $\Z_r \wr S_n$ and the (appropriately shifted) integer points in the fundamental parallelepipeds for the cones over the $\Delta_\pi$.
Note that this bijection encodes $I_{r,n}$ as the integer points in the fundamental parallelepiped for $\cone(\Delta_{\Id})$, where $\Id$ denotes the identity element in $S_n$.
Thus
\[
  \sigma_{\cone(\Delta_\pi)}(z_0,\ldots,z_n) 
  = \sum_{(\rho,\s)\in I_{r,n}}  \frac{\displaystyle \prod_{j\in \Des(\pi)} z_0z_{\pi(1)}z_{\pi(2)}\cdots z_{\pi(j)} \prod_{j\in \NNeg((\rho,\s)^{-1})} z_0z_{\pi(1)}z_{\pi(2)}\cdots z_{\pi(j)} }{\displaystyle (1-z_0)\prod_{j=1}^n \left(1-z_0^r z_{\pi(1)}^r \cdots z_{\pi(j)}^r \right) }  \, .
\]
This completes our proof, since from our triangulation it follows that
\[
\sigma_{\cone ([0,1]^n)}(z_0,\ldots,z_n) = \sum_{\pi\in S_n} \sigma_{\cone(\Delta_\pi)}(z_0,\ldots,z_n) \, . \qedhere
\]
\end{proof}

\begin{proof}[Proof of Theorem~\ref{wreathnegBthm}]
Setting $t:=z_0$ and $q:=z_1=\cdots =z_n$ in Theorem~\ref{wreathcubeBthm} yields our desired form on the left-hand side of our identity, while the denominator of the right-hand side uniformly becomes 
\[
(1-t) \prod_{ j=1 }^n \left( 1 - t^r q^{jr} \right) .
\]
Each element $\displaystyle (\rho,\s)\pi\in \bigcup_{\pi\in S_n} I_{r,n}\pi$ contributes to the numerator on the right-hand side of our identity a summand of
\[
\prod_{j\in \Des(\pi)}tq^j \prod_{j\in \NNeg([(\rho,\s)\pi ]^{-1})} tq^j .
\]
Because $\Des(\pi)=\Des_A((\rho,\s)\pi)$, it follows that 
\[
\prod_{j\in \Des(\pi)}tq^j \prod_{j\in \NNeg([(\rho,\s)\pi ]^{-1})} tq^j = t^{\ndes((\rho,\s)\pi)}q^{\nmaj((\rho,\s)\pi)} ,
\]
hence our proof is complete.
\end{proof}

The following is our multivariate extension of Theorem~\ref{wreathflagBthm}.

\begin{theorem}\label{wreathflagcubeBthm}
\begin{align*}
  \sum_{ k \ge 0 } \prod_{ j=1 }^n [k+1]_{ z_j } \, z_0^k \ = \ 
  \sum_{ (\pi,\s)\in \Z_r\wr S_n }  & \frac{ \displaystyle  \prod_{\substack{j\in \Des(\pi) \\ a_j^\s=0 }}z_0^rz_{\pi(1)}^rz_{\pi(2)}^r\cdots z_{\pi(j)}^r \prod_{j=1}^nz_0^{a_j^\s}z_{\pi(1)}^{a_j^\s}z_{\pi(2)}^{a_j^\s}\cdots z_{\pi(j)}^{a_j^\s}  }{ \displaystyle (1-z_0)\prod_{ j=1 }^n \left( 1 - z_0^r \, z_{\pi(1) }^{r} z_{\pi(2) }^{r} \cdots z_{\pi(j)}^{r} \right) } \, .  
\end{align*}
\end{theorem}

\begin{proof}
We begin again with the triangulation of $\cone ([0,1]^n)$ by the set of cones 
$
\left\{ \cone(\Delta_\pi): \, \pi \in S_n \right\}
$
found in the proof of Theorem~\ref{Atheorem}.
As in our previous proof, for $\overline{\cone(\Delta_\pi)}$ we use the non-unimodular ray generators 
\[
\e_0, r(\e_0+\e_{\pi(1)}), r(\e_0+\e_{\pi(1)}+\e_{\pi(2)}), \ldots, r(\e_0+\e_{\pi(1)}+\cdots + \e_{\pi(n)}) \, .
\]
However, in this proof we use the technique of shifting integer points off of the boundary, discussed in Section~\ref{sec:unim}.
Hence we represent every integer point $\p$ in the fundamental parallelepiped for $\cone(\Delta_{\pi})$ uniquely using a coefficient vector $\alpha\in \{0,1,2,\ldots,r-1\}^n$ in the sum
\[
\p=\sum_{j=1}^n \alpha_j(\e_0+\e_{\pi(1)}+\cdots + \e_{\pi(j)}) + \sum_{\substack{j\in \Des(\pi)\\ \alpha_j=0}} r(\e_0+\e_{\pi(1)}+\cdots + \e_{\pi(j)}) \, .
\]

We may then associate to the point $\p$ the element $(\pi,\s)\in \Z_r\wr S_n$ where $\pi$ is the same as the index on $\cone(\Delta_\pi)$ and $\s$ is defined by $a_j^\s=\alpha_j$.
This bijectively relates $\Z_r\wr S_n$ to the (possibly shifted) integer points in the fundamental parallelepipeds of the cones over the $\Delta_\pi$'s.
Thus
\[
  \sigma_{\cone(\Delta_\pi)}(z_0,\ldots,z_n) 
  = \sum_{(\pi,\s)\in \Z_r\wr S_n}  \frac{\displaystyle \prod_{\substack{j\in \Des(\pi)\\ a_j^\s=0}} z_0^rz_{\pi(1)}^rz_{\pi(2)}^r\cdots z_{\pi(j)}^r \prod_{j=1}^n z_0^{a_j^\s}z_{\pi(1)}^{a_j^\s}z_{\pi(2)}^{a_j^\s}\cdots z_{\pi(j)}^{a_j^\s} }{\displaystyle (1-z_0)\prod_{j=1}^n \left(1-z_0^rz_{\pi(1)}^r\cdots z_{\pi(j)}^r \right) }  \, .
\]
(Note that in the summand on the right-hand side, $\pi$ is fixed while $\s$ varies.)

This completes our proof, since from our triangulation it follows that
\[
\sigma_{\cone ([0,1]^n)}(z_0,\ldots,z_n) = \sum_{\pi\in S_n} \sigma_{\cone(\Delta_\pi)}(z_0,\ldots,z_n) \, . \qedhere
\]
\end{proof}

\begin{remark}\label{rem:descent}
For the proof of Theorem~\ref{wreathflagBthm}, we need to understand the causes of descents in elements of wreath products.
Let $(\pi,\s)\in \Z_r\wr S_n$.
A descent in position $j$ of $(\pi,\s)$ can arise for one of three reasons:
\begin{itemize}
\item \emph{color change:} $c_j<c_{j+1}$, or
\item \emph{standard descent:} $\s_{j}=\s_{j+1}$ and $j\in \Des(\pi)$, or
\item \emph{zero descent:} $j=1$ and $c_1\neq 0$.
\end{itemize}
For example, in $[1^1\, 2^3 \, 5^0\, 3^0\, 4^1 \, 6^0]$, there are color-change descents in positions $1$ and $4$, a standard descent in position $3$, and a zero descent in position $0$.
Descents in position $0$ are precisely those called zero descents, and hence type-$A$ descents arise only from color change and standard descents.

Regarding color-change descents, consider the partial sums $A_k=\sum_{j=k}^na_j^{\s}$ of color changes.
We have that $a_j^{\s}\leq r-1$ and that for all $j$, we obtain one descent for each $k$ such that $A_k\geq lr$ and $A_{k+1}< lr$ for some fixed multiple of $r$.
In less formal terms, as we read in window notation from right to left, each time the partial sum of color changes accrues an additional $r$, that forces another color-change descent.
If $A_1=\ch(\s)$ is a multiple of $r$, then $c_1=0$, and hence there is no zero descent.
On the other hand, if $A_1=\ch(\s)$ is not a multiple of $r$, then this implies $c_1\neq 0$, which creates a zero descent.
Standard descents arise when $a_j^{\s}=0$, in which case a descent in position $j$ is controlled completely by the descent structure of $\pi$.
\end{remark}

\begin{proof}[Proof of Theorem~\ref{wreathflagBthm}]
Setting $t:=z_0$ and $q:=z_1=\cdots =z_n$ in Theorem~\ref{wreathflagcubeBthm} yields our desired form on the left-hand side of our identity, while the denominator of the right-hand side uniformly becomes 
\[
(1-t) \prod_{ j=1 }^n \left( 1 - t^r q^{rj} \right) .
\]
Each element $(\pi,\s)\in \Z_r \wr S_n$ contributes to the numerator on the right-hand side of our identity a summand of
\[
\prod_{\substack{j\in \Des(\pi)\\ a_j^\s=0}} t^rq^{rj} \prod_{j=1}^n t^{a_j^\s}q^{j a_j^\s} .
\]
Therefore, our proof will be complete once we prove that
\begin{equation}\label{wreathfdessignrep}
\fdes(\pi,\s)=\sum_{\substack{j\in \Des(\pi)\\ a_j^\s=0}}r+\sum_{j=1}^na_j^\s
\end{equation}
and
\begin{equation}\label{wreathfmajorsignrep1}
\fmajor(\pi,\s)=\sum_{\substack{j\in \Des(\pi)\\ a_j^\s=0}}rj+\sum_{j=1}^nja_j^\s \, .
\end{equation}

As an example, consider the element $[1^1\, 2^3 \, 5^0\, 3^0\, 4^1 \, 6^0]\in \Z_4\wr S_6$, for which $\Des_A(\pi,\s)=\{1,3,4\}$, $\Des(\pi)=\{3\}$, and $a^\s=(2,3,0,3,1,0)$.
Then we see that $\fdes(\pi,\s)$ is obtained as
\[
4\cdot \des_A(\pi,\s)+c_1=4\cdot 3 + 1 = 13 = 4 + 9 = \sum_{\substack{j\in \Des(\pi)\\ a_j^\s=0}}4+\sum_{j=1}^6a_j^\s \, ,
\]
while $\fmajor(\pi,\s)$ is obtained as both
\[
4\cdot \major_A(\pi,\s)+\col(\pi,\s)=4\cdot 8 + 5 = 37
\]
and
\[
37= 4\cdot 3 + 2+3\cdot 2 + 3\cdot 4 + 5 = \sum_{\substack{j\in \Des(\pi)\\ a_j^\s=0}}4j+\sum_{j=1}^6a_j^\s j \, .
\]

To prove \eqref{wreathfdessignrep}, we build upon Remark~\ref{rem:descent} to investigate the relationship between the values $a_j^{\s}$ and type-$A$ descents.
We must show that 
\[
r\cdot \des_A(\pi,\s) + c_1 = \sum_{\substack{j\in \Des(\pi)\\ a_j^\s=0}}r+\sum_{j=1}^na_j^\s \, .
\]
Following Remark~\ref{rem:descent}, we observe that $\lfloor \frac{\ch(\s)}{r} \rfloor$ is equal to the number of color-change descents in $(\pi,\s)$, which are all type-$A$ descents.
Thus
\[
  \left(\sum_{j=1}^na_j^\s\right)-c_1=r\cdot\left\lfloor \frac{\ch(\s)}{r} \right\rfloor 
\]
is equal to the number of color-change descents multiplied by $r$.
Similarly, Remark~\ref{rem:descent} implies that the number of standard descents multiplied by $r$ (also type-$A$ descents) is given by $\displaystyle \sum_{\substack{j\in \Des(\pi)\\ a_j^\s=0}}r$.
The equality in \eqref{wreathfdessignrep} follows immediately.
To prove \eqref{wreathfmajorsignrep1}, we must show that
\[
r\cdot \major_A(\pi,\s)+\col(\pi,\s)=\sum_{\substack{j\in \Des(\pi)\\ a_j^\s=0}}rj+\sum_{j=1}^na_j^\s j \, .
\]
It follows from Remark~\ref{rem:descent} that $\displaystyle \sum_{\substack{j\in \Des(\pi)\\ a_j^\s=0}}rj$ is equal to the contribution given by standard descents to the type-$A$ major index.
We are left to consider the contribution of color-change descents, hence what remains to be shown is that
\begin{equation}\label{lasteqn}
\sum_{j=1}^na_j^\s j = \col(\pi,\s)+\sum_{\substack{j\in \Des(\pi,\s) \\ a_j^{\s}\neq 0 }} rj \, ,
\end{equation}
which we prove as follows.
\begin{align*}
\sum_{j=1}^nja_j^{\s} & = \sum_{i=1}^n\sum_{j=i}^na_j^{\s} \\
& = \sum_{i=1}^n\left[c_i + r\cdot \#\left\{j\geq i : j\in \Des(\pi,\s)\text{ arising from a color change} \right\} \right] \\
& = \sum_{i=1}^nc_i + \sum_{i=1}^nr\cdot \#\left\{j\geq i : j\in \Des(\pi,\s)\text{ arising from a color change} \right\} \\
& = \col(\pi,\s) + \sum_{\substack{j\in \Des(\pi,\s)\\ a_j^{\s}\neq 0}}jr \, .
\end{align*}
The key observation in the above sequence of equalities is that
\[
\sum_{j=i}^na_j^{\s}=c_i + r\cdot \#\left\{j\geq i : j\in \Des(\pi,\s)\text{ arises from a color change} \right\} \, ,
\]
which follows from the discussion in Remark~\ref{rem:descent}.
\end{proof}

\begin{example}\label{ex:signchange}
To illustrate the key observation at the end of the proof above, consider an arbitrary $(\pi,\s)\in \Z_5\wr S_9$ with the color vector $c=(4,1,2,3,0,1,1,3,1)$.
Note that there are four type-$A$ descents in $(\pi,\s)$ caused by color changes, with color-change descent positions 7, 5, 3, and 2.
The color-change vector for $c$ is $a=(3,4,4,3,4,0,3,2,1)$, where the right-hand $1$ accounts for $c_9-c_{10}$ where $c_{10}$ is by definition $0$.
When $i=4$, we see that  
\begin{align*}
\sum_{j=4}^na_j^{\s} & = 3+4+0+3+2+1 \\
& = 13 \\
& = 3 + 5\cdot 2 \\
& = c_4 + r\cdot \#\left\{j\geq 4 : j\in \Des(\pi,\s)\text{ arises from a color change} \right\} \, .
\end{align*}
\end{example}

The proofs of Theorems~\ref{wreathcubeBthm} and~\ref{wreathflagcubeBthm} together yield a bijective proof of the equidistribution of the pairs of statistics $(\ndes,\nmaj)$ and $(\fdes,\fmajor)$ for $\Z_r\wr S_n$.
As far as we know, this bijection is new.

\begin{corollary}\label{cor:wreathbijection}
\[
\sum_{(\pi,\s)\in \Z_r \wr S_n }t^{\ndes(\pi,\s)}q^{\nmaj(\pi,\s)}  =\sum_{(\pi,\s)\in \Z_r \wr S_n }t^{\fdes(\pi,\s)}q^{\fmajor(\pi,\s)} .
\]
\end{corollary}

\begin{proof}
Our proof relies on the indexing of integer points in fundamental parallelepipeds for $\cone(\Delta_\pi)$ found in the proofs of Theorems~\ref{wreathcubeBthm} and~\ref{wreathflagcubeBthm}.
To the element $(\rho,\s)\pi\in \bigcup_{\pi\in S_n}I_{r,n}\pi$ we associated the integer point
\[
\p=\sum_{j\in \Des(\pi)} (\e_0+\e_{\pi(1)}+\cdots + \e_{\pi(j)}) + \sum_{j=1}^n \alpha_j(\e_0+\e_{\pi(1)}+\cdots + \e_{\pi(j)}) \, 
\]
where $\alpha_j=k$ if and only if $j$ has multiplicity $k$ in $\NNeg((\rho,\s)^{-1})=\NNeg([(\rho,\s)\pi]^{-1})$.
Rewriting $\p$ as
\[
\p=\sum_{j=1}^n \beta_j(\e_0+\e_{\pi(1)}+\cdots + \e_{\pi(j)}) + \sum_{\substack{j\in \Des(\pi)\\ \beta_j=0}} r(\e_0+\e_{\pi(1)}+\cdots + \e_{\pi(j)}) \, .
\]
we associated to $\p$ the element $(\pi,\s)\in \Z_r\wr S_n$ where $\pi$ is the same as the index on $\cone(\Delta_\pi)$ and $\s$ is defined by $a_j^\s=\beta_j$.
This yields an explicit bijection from $\Z_r\wr S_n$ to itself that preserves the pairs of statistics $(\ndes,\nmaj)$ and $(\fdes,\fmajor)$.
\end{proof}

\begin{example}
Let $r=4$ and $n=6$, and consider the element $[1^3 \, 4^0 \, 2^1 \, 3^0 \, 6^2 \, 5^1]$ with the $\NNeg$ set of its inverse equal to $\{1,2,2,2,5,5,5,6,6\}$.
Our goal is to find the element in $\Z_4\wr S_6$ paired with this element under our bijection.
We first encode the element as an integer point, which requires using $\pi=[1\,6\, 3\, 5\, 2\, 4]$ and $\alpha_1=1$, $\alpha_2=3$, $\alpha_3=\alpha_4=0$, $\alpha_5=3$, and $\alpha_6=2$; note that $\Des(\pi)=\{2,4\}$.
Thus, writing the two summands arising from the descent positions in $\pi$ first in the first sum, we have that the element is encoded by
\begin{align*}
\p  = & \phantom{=} \e_0+\e_1+\e_6 + \\
& \phantom{=} \e_0+\e_1+\e_6+\e_3+\e_5 +  \\
& \phantom{=} \e_0+\e_1 + \\
& \phantom{=} 3(\e_0+\e_1+\e_6) + \\
& \phantom{=} 3(\e_0+\e_1+\e_6+\e_3+\e_5+\e_2) + \\
& \phantom{=} 2(\e_0+\e_1+\e_6+\e_3+\e_5+\e_2+\e_4) \\
 = & \phantom{=} \e_0+\e_1 + \\
& \phantom{=} \e_0+\e_1+\e_6+\e_3+\e_5 + \\
& \phantom{=} 3(\e_0+\e_1+\e_6+\e_3+\e_5+\e_2) \\
& \phantom{=} 2(\e_0+\e_1+\e_6+\e_3+\e_5+\e_2+\e_4) \\
& \phantom{=} 4(\e_0+\e_1+\e_6) \, .
\end{align*}
Note that in the final sum above, we have merged our summands arising from $\Des(\pi)$ in the first sum into the others to obtain a representation of $\p$ where the last of our summands has a coefficient of $4$.
This $4$ arises in one of the descent positions for $\pi$, corresponding to $\beta_2=0$.
Hence, our second encoding vector is $\beta=(1,0,0,1,3,2)$.
Thus, we recover our new element of $\Z_4\wr S_6$ by setting $a^{\s}=\beta$, obtaining
\[
(\pi,\s)=[1^3\, 6^2\, 3^2\, 5^2\, 2^1\, 4^2] \, .
\]
Finally, observe that for our original element $[1^3 \, 4^0 \, 2^1 \, 3^0 \, 6^2 \, 5^1]$, we have that
\[
(\ndes,\nmaj)= (11,40) \, ,
\]
while for the element $[1^3\, 6^2\, 3^2\, 5^2\, 2^1\, 4^2]$, we have that
\[
(\fdes,\fmajor)= (11,40) \, ,
\]
as desired.
\end{example}


\subsection{Identities involving $(rk+1)^n$}

In \cite[Theorem 9]{chowmansourcarlitz}, Chow and Mansour provide an Euler--Mahonian distribution for wreath products using Steingr{\'{\i}}msson's descent statistics and a new flag major index.
Their identity is a generalization of a result due to Chow--Gessel which we discuss in Section~\ref{bigbsection}.
In this section, we state a similar Euler--Mahonian distribution for the descent statistic and flag major index given in Definitions~\ref{BnDes} and~\ref{wreathflagdef}.
By combining Theorem~\ref{chowmansourflagthm} below and \cite[Theorem 9]{chowmansourcarlitz}, we see that the pairs 
\[ 
(\text{Steingr{\'{\i}}msson's descent statistic},\text{Chow--Mansour's flag major index})
\]
and
\[
(\des,\fmajor)
\]
are equidistributed over $\Z_r\wr S_n$.

\begin{theorem}\label{chowmansourflagthm}
\[
  \sum_{ k \ge 0 } [rk+1]_q^n \, t^k
  \ = \ \frac{\sum_{(\pi,\s)\in \Z_r\wr S_n}t^{\des(\pi,\s)}q^{\fmajor(\pi,\s)} }{ \prod_{ j=0 }^n \left( 1 - t q^{rj} \right) } \, .
\]
\end{theorem}

We obtain in Theorem~\ref{wreathnewBthm} below a multivariate generalization of this bivariate identity.

\begin{remark}At the end of \cite{chowmansourcarlitz}, Chow and Mansour indicate that having a Hilbert-series interpretation of \cite[Theorem 9]{chowmansourcarlitz} is an open problem.
The proof of Theorem~\ref{wreathnewBthm} provides such an interpretation, after taking into account Remark~\ref{hilbertseries}.
\end{remark}

\begin{remark}
In \cite[Equation (8.1)]{biagiolizengwreath}, Biagioli--Zeng obtain a wreath product version of Theorem~\ref{chowgesselthm},
a result due to Chow--Gessel.
The authors do not at this time know of a way to obtain this identity using polyhedral geometry.
\end{remark}

\subsection{Multivariate identities}

Our multivariate generalization of Theorem~\ref{chowmansourflagthm} is the following.

\begin{theorem}\label{wreathnewBthm}
\begin{align*}
  \sum_{ k \ge 0 } \prod_{ j=1 }^n [rk+1]_{ z_j } \, z_0^k \ = \ 
  \sum_{ (\pi,\s)\in \Z_r\wr S_n }  \frac{ \displaystyle z_0^{\left\lceil  \ch(\s)/r \right\rceil } \prod_{j=1}^nz_{\pi(1)}^{a_j^\s}z_{\pi(2)}^{a_j^\s}\cdots z_{\pi(j)}^{a_j^\s}  \prod_{\substack{j\in \Des(\pi) \\ a_j^\s=0 }}z_0z_{\pi(1)}^rz_{\pi(2)}^r\cdots z_{\pi(j)}^r }{ \displaystyle \prod_{ j=0 }^n \left( 1 - z_0 \, z_{\pi(1) }^{r} z_{\pi(2) }^{r} \cdots z_{\pi(j)}^{r} \right) } \, .  
\end{align*}
\end{theorem}

\begin{proof}
As in our previous proofs, this proof proceeds in two stages.
We first triangulate the cube $[0,r]^n$ into a disjoint union of simplices, then set up an indexing system for the integer points in the fundamental parallelepipeds for the cones over these simplices.
Second, we bijectively associate the elements of $\Z_r \wr S_n$ with these integer points in a way that allows us to recover, in our subsequent proof of Theorem~\ref{chowmansourflagthm}, the descent and flag major index statistics from these integer points.

We begin by triangulating $[0,r]^n$ into the disjoint simplices
\[
  \Delta_\pi := \left\{ \x \in \R^n :
     \begin{array}{ll}
       0 \le x_{ \pi(n) } \le x_{ \pi(n-1) } \le \dots \le x_{ \pi(1) } \le r, \\
       x_{ \pi(j+1) } < x_{ \pi(j) } \text{ if } j \in \Des(\pi)
     \end{array}
  \right\} 
\]
(one for each $\pi \in S_n$).
As before, the strict inequalities determined by the descent set of $\pi$ ensures that this triangulation is disjoint.

Unlike the cones produced by coning over the simplices in our triangulation of $[0,1]^n$, the cones arising from this triangulation of $[0,r]^n$ are not unimodular.
By Lemma \ref{genconelemma}, the integer-point transform of $\cone(\Delta_{\pi})$ can be expressed as a rational function where the denominator has the form 
\[
\prod_{ j=0 }^n \left( 1 - z_0 \, z_{\pi(1) }^{r} z_{\pi(2) }^{r} \cdots z_{\pi(j)}^{r} \right) \, ,
\]
i.e., where the displayed exponent vectors are the ray generators for this cone.
As the determinant of the matrix formed by the ray generators of $\overline{\cone(\Delta_\pi)}$ is $r^n$, there are $r^n$ integer points in the fundamental parallelepiped of $\overline{\cone(\Delta_\pi)}$.
It is a straightforward observation that there are $r^n$ such integer points formed by taking linear combinations of the ray generators for the cone with coefficients from the set $\left\{0, \frac 1 r,\ldots, \frac{ r-1 }{ r } \right\}$.
We will use the following notation to denote such an integer point; for $\alpha_j \in \left\{0, \frac 1 r,\ldots, \frac{ r-1 }{ r } \right\}$ where $j=0,\ldots, n$,
\[
\p=\alpha_0\e_0 + \sum_{j=1}^n \alpha_j\left( \e_0 + r\e_{\pi(1)} + r\e_{\pi(2)} + \cdots + r\e_{\pi(j)} \right) \, .
\]
Observe that because $\p$ is an integer point, the value of $\alpha_0$ is determined by the condition that the coefficient of $\e_0$, $\sum_{j=0}^n\alpha_j$, be an integer.

To determine the numerator of $\sigma_{\cone(\Delta_\pi)}(z_0,\ldots,z_n)$, as in our earlier situations dealing with unimodular cones, we must shift some integer points off of the boundary of $\Pi_{\overline{\cone(\Delta_\pi)}}$, specifically those that are not contained in $\Pi_{\cone(\Delta_\pi)}$.
When $\p$ is contained in a given facet indexed by $j\in \Des(\pi)$, then we must shift $\p$ by the minimal ray generator opposite that facet, namely $(\e_0+r\e_{\pi(1)}+r\e_{\pi(2)}+\cdots +r\e_{\pi(j)})$.
Such a point $\p$, when written in the form displayed above, is contained in such a facet precisely when $\alpha_j=0$.
Thus, each such $\p$ must be shifted from $\Pi_{\overline{\cone(\Delta_\pi)}}$ to $\Pi_{\cone(\Delta_\pi)}$ by the vector 
\[
\sum_{\substack{j\in \Des(\pi) \\ \alpha_j=0}}(\e_0+r\e_{\pi(1)}+r\e_{\pi(2)}+\cdots +r\e_{\pi(j)}) \, ,
\]
yielding the point 
\[
\alpha_0\e_0 + \sum_{j=1}^n \alpha_j\left( \e_0 + r\e_{\pi(1)} + r\e_{\pi(2)} + \cdots + r\e_{\pi(j)} \right) + \sum_{\substack{j\in \Des(\pi) \\ \alpha_j=0}}(\e_0+r\e_{\pi(1)}+r\e_{\pi(2)}+\cdots +r\e_{\pi(j)})
\]
in $\Pi_{\overline{\cone(\Delta_\pi)}} \, .$
Hence
\begin{align*}
 \sigma_{ \cone(\Delta_{ \pi }) } (z_0, z_{ 1 }, & \dots, z_{ n }) \ = \
 \sum_{\alpha \in \{0,\frac{1}{r},\ldots,\frac{r-1}{r}\}^n} \frac{ \displaystyle z_0^{ \sum_{j=0}^n \alpha_j } \prod_{j=1}^nz_{\pi(1)}^{r\alpha_j}z_{\pi(2)}^{r\alpha_j}\cdots z_{\pi(j)}^{r\alpha_j}  \prod_{\substack{j\in \Des(\pi) \\ \alpha_j=0 } }z_0z_{\pi(1)}^rz_{\pi(2)}^r\cdots z_{\pi(j)}^r  }{\displaystyle \prod_{ j=0 }^n \left( 1 - z_0 \, z_{\pi(1) }^{r} z_{\pi(2) }^{r} \cdots z_{\pi(j)}^{r} \right) } \, .
\end{align*}

We now associate to the element $(\pi,\s)\in \Z_r\wr S_n$ the integer point in $\Pi_{\cone(\Delta_\pi)}$ with $\alpha_j:=\frac{a_j^\s}{r}$ for $j=1,\ldots,n$, i.e. the point
\[
\alpha_0 \e_0 + \sum_{j=1}^n \frac{a_j^\s}{r}(\e_0 + r\e_{\pi(1)} + r \e_{\pi(2)}+ \cdots + r \e_{\pi(j)}) \, + \sum_{\substack{j\in \Des(\pi) \\ a_j^\s=0}} (\e_0 + r\e_{\pi(1)} + r \e_{\pi(2)}+ \cdots + r \e_{\pi(j)}) \, ,
\]
where $\alpha_0$ is determined by the condition that $\alpha_0+\sum_{j=1}^n \frac{a_j^\s}{r}$ be an integer.
Through this association, the set of ``color" vectors $\{1,\omega,\omega^2,\ldots,\omega^{r-1}\}^n$ parametrizes the integer points in the fundamental parallelepiped for $\cone(\Delta_\pi)$.
This parametrization is bijective, and the coefficient of $\e_0$ in the first two terms of the sum above is equal to both $\lceil \frac{\ch(\s)}{r} \rceil $ and $\sum_{j=0}^n \alpha_j$.
Thus, our proof is complete following the observation that
\[ 
\sigma_{\cone([0,r]^n)}(z_0,\ldots,z_n) = \sum_{\pi\in S_n} \sigma_{\cone(\Delta_\pi)}(z_0,\ldots,z_n) \, . \qedhere
\]
\end{proof}

\begin{proof}[Proof of Theorem \ref{chowmansourflagthm}]
Setting $t:=z_0$ and $q:=z_1=\cdots =z_n$ in Theorem~\ref{wreathnewBthm} yields our desired form on the left-hand side, while the denominator of the right-hand side uniformly becomes $ \prod_{ j=0 }^n \left( 1 - t q^{rj} \right)$.
Each element $(\pi,\s)\in \Z_r\wr S_n$ contributes a summand of
\[
t^{\left\lceil  \ch(\s)/r \right\rceil } \prod_{j:a_j^\s\neq 0}q^{a_j^{\s}j}  \prod_{\substack{j\in \Des(\pi) \\ a_j^\s=0 }}tq^{rj}
\]
to the numerator of the right-hand side.
Hence, we need to prove 
\begin{equation}\label{wreathdessignrep}
\des(\pi,\s)=\left\lceil  \tfrac{ \ch(\s) }{ r } \right\rceil + \#\{ j\in \Des(\pi): a_j^\s=0\} 
\end{equation}
and
\begin{equation}\label{wreathfmajorsignrep}
\fmajor(\pi,\s)=\sum_{j=1}^na_j^\s j + \sum_{\substack{j\in \Des(\pi)\\ a_j^\s=0}}rj \, .
\end{equation}
Observe that \eqref{wreathfmajorsignrep} is identical to \eqref{wreathfmajorsignrep1}, which was proved earlier.

As discussed in Remark~\ref{rem:descent}, a descent in position $j$ of $(\pi,\s)$ can be a color-change descent, a standard descent, or a zero descent.
If $\tfrac{ \ch(\s) }{ r }$ is an integer, then there is no zero descent, and there are $\tfrac{ \ch(\s) }{ r }$ color-change descents.
If $\tfrac{ \ch(\s) }{ r }$ is not an integer, there are $\lfloor \tfrac{ \ch(\s) }{ r } \rfloor$ color-change descents and a zero descent, which contribute a total of $\lceil \tfrac{ \ch(\s) }{ r } \rceil$ descents.
Further, the standard descents are counted by $\#\{ j\in \Des(\pi): a_j^\s=0\}$.
The equality in \eqref{wreathdessignrep} follows immediately from these observations.
\end{proof}



\section{Type $B$}\label{bigbsection}


In this section we state our main results from Section~\ref{wreathsection} in the special case of hyperoctahedral groups.
We also prove a new multivariate identity given in Theorem~\ref{Btheorem}.


\subsection{Identities involving $(k+1)^n$}

In \cite{adinbrentiroichman}, Adin, Brenti, and Roichman introduced several pairs of statistics on $B_n$; these were the inspiration for the statistics considered in Section~\ref{wreathsection}.
For the interested reader, we state the original identities of Adin--Brenti--Roichman and our multivariate identities generalizing them.
The statistics arising here are special cases of those defined in Section~\ref{wreathsection}.

\begin{theorem}[Adin--Brenti--Roichman]\label{negBthm}
\[
  \sum_{ k \ge 0 } [k+1]_q^n \, t^k
  \ = \ \frac{ \sum_{(\pi,\s) \in B_n } t^{ \ndes(\pi,\s) } q^{ \nmaj(\pi,\s) } }{ (1-t)\prod_{j=1}^n (1-t^2q^{2j}) } \, .
\]
\end{theorem}

\begin{theorem}[Adin--Brenti--Roichman]\label{flagBthm}
\[
  \sum_{ k \ge 0 } [k+1]_q^n \, t^k
  \ = \ \frac{ \sum_{(\pi,\s) \in B_n } t^{ \fdes(\pi,\s) } q^{ \fmajor(\pi,\s) } }{ (1-t)\prod_{j=1}^n (1-t^2q^{2j}) } \, .
\]
\end{theorem}

Note that in the original work of Adin--Brenti--Roichman, the flag and negative statistics were defined using the natural order; in that context, the flag major index was denoted by $\mathrm{fmaj}$ rather than $\fmajor$.
However, Adin--Brenti--Roichman point out in \cite[p.\ 218]{adinbrentiroichman} that either order can be used to obtain Theorems~\ref{negBthm}~and~\ref{flagBthm}.
Our generalizations in type $B$ are the following corollaries of Theorems~\ref{wreathcubeBthm}~and~\ref{wreathflagcubeBthm}, again using notation from Section~\ref{wreathsection}.

\begin{corollary}\label{cubeBthm}
\begin{align*}
 \sum_{k \ge 0 } \prod_{j=1}^n [k+1]_{ z_j }  z_0^k \ = \
\sum_{\pi\in S_n} \sum_{(\rho,\s)\in I_{2,n}}  &  \frac{\displaystyle \prod_{j\in \Des(\pi)} z_0z_{\pi(1)}z_{\pi(2)}\cdots z_{\pi(j)} \prod_{j\in \NNeg((\rho,\s)^{-1})} z_0z_{\pi(1)}z_{\pi(2)}\cdots z_{\pi(j)} }{\displaystyle (1-z_0)\prod_{j=1}^n \left(1-z_0^2z_{\pi(1)}^2\cdots z_{\pi(j)}^2 \right) } \, .
\end{align*}
\end{corollary}

\begin{corollary}\label{flagcubeBthm}
\begin{align*}
  \sum_{ k \ge 0 } \prod_{ j=1 }^n [k+1]_{ z_j } z_0^k \ = \
  \sum_{ (\pi,\s)\in B_n }  & \frac{ \displaystyle  \prod_{\substack{j\in \Des(\pi) \\ a_j^\s=0 }}z_0^2z_{\pi(1)}^2z_{\pi(2)}^2\cdots z_{\pi(j)}^2 \prod_{j:a_j^\s=1}z_0z_{\pi(1)}z_{\pi(2)}\cdots z_{\pi(j)}  }{ \displaystyle (1-z_0)\prod_{ j=1 }^n \left( 1 - z_0^2 \, z_{\pi(1) }^{2} z_{\pi(2) }^{2} \cdots z_{\pi(j)}^{2} \right) } \, .  
\end{align*}
\end{corollary}

In the original work of Adin, Brenti, and Roichman \cite{adinbrentiroichman}, it was left as an open question to give a bijective proof in type $B$ of the equidistribution of the pairs of statistics $(\ndes,\nmaj)$ and $(\fdes,\fmajor)$; a combinatorial proof in type $B$ leading to an implicit bijection was given by Lai and Petersen in \cite{petersenlai}.
When restricted to type $B$, the proof of Corollary~\ref{cor:wreathbijection} also produces such a bijection.


\subsection{Identities involving $(2k+1)^n$}

Recall Definition~\ref{BnNatDes} which introduced $\NatDes(\pi,\s)$ and $\natdes(\pi,\s)$.
For an element $(\pi,\s)\in B_n$, the \emph{naturally ordered major index} is
\[
\natmaj (\pi, \s):= \sum_{i\in \NatDes(\pi,\s)}i \, .
\]
Further, for an element $(\pi,\s)\in B_n$, we write $\nega(\pi,\s):=\col(\pi,\s)$, where $\col(\pi,\s)$ is given in Definition~\ref{typeAwreathstats}, to emphasize that this statistic counts the number of negative signs in the window for $(\pi,\s)$.
Chow and Gessel \cite[Equation (26)]{chowgessel} proved the following hyperoctahedral analogue of Theorem~\ref{macmahoncarlitz}:

\begin{theorem}[Chow--Gessel]\label{chowgesselthm}
\[
  \sum_{ k \ge 0 } \left( [k+1]_q + s \, [k]_q \right)^n t^k
  \ = \ \frac{ \sum_{ \pi \in S_n, \, \s \in \left\{ \pm 1 \right\}^n } s^{ \nega(\pi,\s) } t^{ \natdes(\pi, \s) } q^{ \natmaj(\pi, \s) } }{ \prod_{ j=0 }^n \left( 1 - t q^j \right) } \, .
\]
\end{theorem}

The special case $q=1$ is due to Brenti \cite[Theorem 3.4]{brentieulerian}.
Chow and Gessel also showed in \cite{chowgessel} how Theorem \ref{chowgesselthm} implies other versions of ``$q$-Eulerian polynomials" of type $B$ involving a flag major index statistic using the natural order, such as the following.

\begin{theorem}[Chow--Gessel]\label{newchowgesselflagthm}
\[
  \sum_{ k \ge 0 }  [2k+1]_q^n t^k
  \ = \ \frac{\sum_{(\pi,\s)\in B_n}t^{\natdes(\pi,\s)}q^{\natfmaj(\pi,\s)} }{ \prod_{ j=0 }^n \left( 1 - t q^{2j} \right) } \, .
\]
\end{theorem}
The statistic $\natfmaj$ used above is defined as follows.

\begin{definition}
Use the order $-n<\cdots<-1<1<\cdots<n$ on $[-n,n]\setminus\{0\}$.
We define the \emph{natural type-$A$ descent set} as
\[
\NatDes \!\mbox{}_A(\pi,\s):= \{i\in [n-1] : \s_i\pi_i > \s_{i+1}\pi_{i+1} \} \,
\]
while the \emph{natural type-$A$ descent statistic} is $\natdes_A(\pi,\s):=\#\NatDes \!\mbox{}_A(\pi,\s)$.
The \emph{natural type-$A$ major index} is defined as
\[
\natmajor_A(\pi,\s):=\sum_{i \in \NatDes_A(\pi,\s)} i \, .
\]
The \emph{natural flag major index} is
\[
\natfmaj(\pi,\s):= 2\cdot \natmajor_A(\pi,\s) + \nega(\pi,\s) \, .
\]
\end{definition}

\begin{remark}
Through this work, $\natdes$ will always refer to the statistic introduced in Definition~\ref{BnNatDes} while $\natdes_A$ will be used to indicate the definition given above.
\end{remark}

While it is observed by Chow--Gessel in \cite{chowgessel} that Theorems~\ref{chowgesselthm} and~\ref{newchowgesselflagthm} are equivalent via a change of variables, the geometric perspective illustrates how these two theorems arise as specializations of two distinct multivariate generating-function identities.


\subsection{Multivariate identities}\label{bsection}

For the type-$B$ generalization of Theorem~\ref{chowgesselthm}, we introduce the variables $z_{\pm j}$ to keep track of the positive/negative $j$th component of a lattice point, respectively, and the variable $s$ to indicate the presence in each coordinate of our point of a negative sign.

\begin{theorem}\label{Btheorem}
\begin{align*}
  \sum_{ k \ge 0 } \prod_{ j=1 }^n \left( [k+1]_{ z_j } + s \, z_{-j}^{ -1 } [k]_{ z_{-j}^{ -1 } } \right) z_0^k \ = \
  \sum_{ (\pi,\s)\in B_n } \displaystyle s^{\nega(\s)} \phantom{,} \frac{\displaystyle \prod_{ j \in \NatDes(\pi, \s) } z_0z_{ \s_{j+1} \pi(j+1) }^{ \s_{j+1} } z_{ \s_{j+2} \pi(j+2) }^{ \s_{j+2} } \cdots z_{\s_n \pi(n)}^{ \s_n } }{ \displaystyle \prod_{ j=0 }^n \left( 1 - z_0 \, z_{ \s_{j+1} \pi(j+1) }^{ \s_{j+1} } z_{ \s_{j+2} \pi(j+2) }^{ \s_{j+2} } \cdots z_{\s_n \pi(n)}^{ \s_n } \right) } \, .  
\end{align*}
\end{theorem}

\begin{proof}
Recall that we use the order 
\[
-n<-n+1<\cdots <-1<1<2<\cdots <n \, .
\]
As in Definition~\ref{Irn}, create the set of increasing elements, denoted $I_{2,n}^{\nat}$, using the natural order above.
It is straightforward from the discussion following Definition~\ref{Irn} to show that
\[
B_n=\bigcup_{\pi\in S_n}I_{2,n}^{\nat}\pi \, .
\]
For each $(\rho, \gamma)\in I_{2,n}^{\nat}$, the first $\nega(\gamma)$ elements of the permutation are negated, with labels $\rho_1,\ldots,\rho_{\nega (\gamma)}$.

For each $(\rho, \gamma)\in I_{2,n}^{\nat}$, define
\[
\Box_{(\rho,\gamma)}:=
\left\{ x\in \R^n:
     \begin{array}{ll}
       0 \le x_{\rho(j)} \le 1 & \text{ if } j\geq \nega(\gamma)+1, \\
       0 < -x_{\rho(j)} \le 1 & \text{ if } j\leq \nega(\gamma) 
     \end{array}
\right\} \, .
\]
It is straightforward to show that 
\[
[-1,1]^n=\bigcup_{(\rho,\gamma)\in I_{2,n}^{\nat}} \Box_{(\rho,\gamma)} \, ,
\]
where this union is disjoint; any point in $[-1,1]^n$ with negative entries in positions $i_1>\cdots>i_k$ is an element of $\Box_{(\rho,\gamma)}$ where $\rho_j=i_j$ and $\gamma$ has $-1$ in precisely the first $k$ entries.

Fix $(\rho,\gamma)\in I_{2,n}^{\nat}$, and for each $\tau\in S_n$ consider the element $(\pi,\s)=(\rho,\gamma)\tau$.
For each such $(\pi,\s)$, set
\[
\Delta_{(\pi,\s)}:=
\left\{ x\in \Box_{(\rho,\gamma)}: 
     \begin{array}{l}
       0\leq \s_1x_{\pi(1)}\leq \cdots \leq \s_nx_{\pi(n)}\leq 1 \\
       \s_jx_{\pi(j)}<\s_{j+1}x_{\pi(j+1)} \text{ if } j\in \NatDes(\pi,\s)
     \end{array}
\right\} \, ,
\]
where $\s_0x_{\pi(0)}=0$.
Thus, the left-most inequality might be strict, while the right-most inequality is never strict.
It follows that
\[
\Box_{(\rho,\gamma)}=\bigcup_{\substack{(\pi,\s)=(\rho,\gamma)\tau \\ \tau\in S_n}} \Delta_{(\pi,\s)} \, ,
\]
where our union is again disjoint.
Observe that this triangulation of $\Box_{(\rho,\gamma)}$ is induced by 
\[
\left\{x_i=x_j: i,j \leq \nega(\gamma) \text{ or } i,j \geq \nega(\gamma)+1  \right\} \cup \left\{x_i=-x_j: i\geq \nega(\gamma)+1 \text{ and } j\leq \nega(\gamma) \right\} \, ,
\]
a sub-arrangment of the type $B$ braid arrangement that intersects $\Box_{(\rho,\gamma)}$ in the same manner as the type $A$ braid arrangement intersects $[0,1]^n$.

For example, given $[2^1 \, 1^1 \, 3^0]=(\rho,\gamma)\in I_{2,n}^{\nat}$, the six elements of $[2^1 \, 1^1 \, 3^0]S_3$ are
\begin{align*}
[2^1 \, 1^1 \, 3^0]\circ [1 \, 2 \, 3] & = [2^1 \, 1^1 \, 3^0]\\
[2^1 \, 1^1 \, 3^0]\circ [2 \, 1 \, 3] & = [1^1 \, 2^1 \, 3^0]\\
[2^1 \, 1^1 \, 3^0]\circ [1 \, 3 \, 2] & = [2^1 \, 3^0 \, 1^1]\\
[2^1 \, 1^1 \, 3^0]\circ [3\,  2 \, 1] & = [3^0 \, 1^1 \, 2^1]\\
[2^1 \, 1^1 \, 3^0]\circ [2 \, 3 \, 1] & = [1^1 \, 3^0 \, 2^1]\\
[2^1 \, 1^1 \, 3^0]\circ [3 \, 1 \, 2] & = [3^0 \, 2^1 \, 1^1] \, ,
\end{align*}
giving rise to $\Box_{[2^1 \, 1^1 \, 3^0]}$ being a union of the six corresponding $\Delta_{(\pi,\s)}$'s shown below:
\begin{align*}
\Delta_{[2^1 \, 1^1 \, 3^0]} & = \{x\in \Box_{[2^1 \, 1^1 \, 3^0]}: 0 < -x_2 \leq -x_1 \leq x_3 \leq 1\} \\
\Delta_{[1^1 \, 2^1 \, 3^0]} & = \{x\in \Box_{[2^1 \, 1^1 \, 3^0]}: 0 < -x_1 < -x_2 \leq x_3\leq 1\} \\
\Delta_{[2^1 \, 3^0 \, 1^1]} & = \{x\in \Box_{[2^1 \, 1^1 \, 3^0]}: 0< -x_2 \leq x_3<-x_1 \leq 1\}\\ 
\Delta_{[3^0 \, 1^1 \, 2^1]} & = \{x\in \Box_{[2^1 \, 1^1 \, 3^0]}: 0\leq x_3<-x_1<-x_2\leq 1\}\\
\Delta_{[1^1 \, 3^0 \, 2^1]} & = \{x\in \Box_{[2^1 \, 1^1 \, 3^0]}: 0<-x_1\leq x_3<-x_2\leq 1\}\\
\Delta_{[3^0 \, 2^1 \, 1^1]} & = \{x\in \Box_{[2^1 \, 1^1 \, 3^0]}: 0\leq x_3<-x_2\leq -x_1\leq 1\} \, .
\end{align*}

A lattice point $\m\in \cone(\Box_{(\rho,\gamma)})$ gets encoded by the monomial 
\[
  z_0^{ m_0 } \prod_{ \s_j = -1 } s \, z_{ -j }^{ -m_j } \prod_{ \s_j = 1 } \, z_j^{ m_j } \, .
\]
Because of the definition of $\Delta_{(\pi,\s)}$, we can use our shifting techniques from Section~\ref{sec:unim} (either technique will suffice in this case) to conclude
\begin{align*}
 \sigma_{ \cone(\Delta_{( \pi, \s) }) } (z_0, z_{ \pm 1 }^{\pm 1}, & \dots, z_{ \pm n }^{\pm 1}, s) \ = \ 
 \frac{\displaystyle s^{\nega(\pi,\s)} \prod_{ j \in \NatDes(\pi, \s) } z_0 \, z_{ \s_{ j+1 } \pi(j+1) }^{ \s_{ j+1 } } z_{ \s_{ j+2 } \pi(j+2) }^{ \s_{ j+2 } } \cdots z_{ \s_n \pi(n) }^{ \s_n } }{\displaystyle \prod_{ j=0 }^n \left( 1 - z_0 \, z_{ \s_{ j+1 } \pi(j+1) }^{ \s_{ j+1 } } z_{ \s_{ j+2 }  \pi(j+2) }^{ \s_{ j+2 } } \cdots z_{ \s_n \pi(n) }^{ \s_n } \right) } \, .
\end{align*}
On the other hand,
\begin{align*}
  &\sigma_{ \cone([-1,1]^n) } (z_0, z_{\pm 1}^{\pm 1}, \dots, z_{\pm n}^{\pm 1}, s ) = \\
  &\qquad \sum_{ k \ge 0 } \prod_{ j=1 }^n \left( s \, z_{ -j }^{-k} + s \, z_{ -j }^{-(k-1)} + \dots + s \, z_{ -j }^{ -1 } + 1 + z_j + z_j^2 + \dots + z_j^k \right) z_0^k \, ,
\end{align*}
and the disjoint triangulations discussed above yield
\[
  \sigma_{ \cone([-1,1]^n) } (z_0, z_{\pm 1}^{\pm 1}, \dots, z_{\pm n}^{\pm 1}, s ) = \sum_{ (\pi,\s) \in B_n } \sigma_{ \cone(\Delta_{\pi,\s}) } (z_0, z_{\pm 1}^{\pm 1}, \dots, z_{\pm n}^{\pm 1}, s ) \, . \qedhere
\]
\end{proof}

\begin{proof}[Proof of Theorem \ref{chowgesselthm}]
Setting $t := z_0$ and $q := z_1 = \dots = z_n = z_{ -1 }^{ -1 } = \dots = z_{ -n }^{ -n }$ in Theorem \ref{Btheorem} gives
\[
  \sum_{ k \ge 0 } \left( [k+1]_q + s \, q [k]_q \right)^n t^k
  = \sum_{ (\pi,\s)\in B_n} s^{\nega(\s)} \, \frac{\displaystyle \prod_{j\in \NatDes(\pi,\s)} t q^{ n-j } }{ \prod_{ j=0 }^n \left( 1 - t q^{ n-j } \right) } \, .
\]
Applying the change of variables $q\mapsto \frac 1 q$ and $t\mapsto tq^n$ finishes the proof.
\end{proof}

Our multivariate generalization of Theorem~\ref{newchowgesselflagthm} is the following, which is a special case of Theorem~\ref{wreathnewBthm}.
Recall from Definition~\ref{wreathdefsignchange} the notation $\ch(\s)$ for the number of color changes in $\s$ and the notation $a_j^\s$ to keep track of where color changes occur.

\begin{theorem}\label{newBtheorem}
\begin{align*}
  &\sum_{ k \ge 0 } \prod_{ j=1 }^n [2k+1]_{ z_j }  z_0^k \ = \ 
  \sum_{ (\pi,\s)\in B_n }  \frac{ \displaystyle z_0^{\left\lceil  \ch(\s)/2 \right\rceil } \prod_{j:a_j^\s=1}z_{\pi(1)}z_{\pi(2)}\cdots z_{\pi(j)}  \prod_{\substack{j\in \Des(\pi) \\ a_j^\s=0 }}z_0z_{\pi(1)}^2z_{\pi(2)}^2\cdots z_{\pi(j)}^2 }{ \displaystyle \prod_{ j=0 }^n \left( 1 - z_0 \, z_{\pi(1) }^{2} z_{\pi(2) }^{2} \cdots z_{\pi(j)}^{2} \right) } \, .  
\end{align*}
\end{theorem}

\begin{remark}
Observe that by specializing Theorem~\ref{newBtheorem} using $t:=z_0$ and $q:=z_1=\cdots =z_n$, we obtain a bivariate generating function identity involving the joint distribution for $(\des,\fmajor)$.
Theorem~\ref{newchowgesselflagthm} follows from this, as the pairs of statistics $(\natdes,\natfmaj)$ and $(\des,\fmajor)$ are equidistributed in $B_n$; this is a consequence of the bijection mapping every permutation $(\pi,\s)\in B_n$ to the permutation where the $\pi(j)$ for $\s_j=-1$ are reversed in order in the window for $(\pi,\s)$, while the $\s$-vector remains the same.

As an example, consider $[2^1\, 4^1\, 5^0\, 1^1 \, 3^1]\in B_5$.
The entries $2$, $4$, $1$, and $3$ correspond to the positions where $\s_j=-1$.
Hence, by reversing the order of these entries, we obtain a new permutation $[3^1\, 1^1\, 5^0\, 4^1 \, 2^1]$, and it is immediate that the descent positions for the new permutation using the natural order are the same as those in the first permutation using our standard order for wreath products.
Observe that for the first permutation we have $(\des,\fmajor)=(2,10)$, and for the second we also have $(\natdes,\natfmaj)=(2,10)$.

Hence, we may conclude that Theorem~\ref{newchowgesselflagthm} follows from the special case of $r=2$ in Theorem~\ref{wreathnewBthm}.
\end{remark}


\section{Type $D$}\label{dsection}

In this section we prove a multivariate identity related to negative statistics on Coxeter groups of type $D$.
One may consider type-$D$ Eulerian polynomials stemming from the signed permutations in $B_n$ with an even number of $-1$'s.
Let
\[
  D_n := \left\{ (\pi, \s) \in B_n : \, \s_1 \cdots \s_n = 1 \right\} .
\]
The definition of $\DNatDes(\pi, \s)$ and $\dnatdes(\pi, \s)$ in type $D$ is analogous to \eqref{Bdescentdef}, except that we now use the convention $\s_0 \pi(0) := - \s_2 \pi(2)$.
Brenti \cite[Theorem 4.10]{brentieulerian} proved that
\begin{equation}\label{deuleriangenfcteq}
  \sum_{ k \ge 0 } \left( \left( 2k+1 \right)^n - 2^{ n-1 } \left( B_n (k+1) - B_n(0) \right)  \right) t^k 
  \ = \ \frac{ \sum_{ (\pi, \s) \in D_n } t^{ \dnatdes(\pi, \s) } }{ (1-t)^{ n+1 } } \, ,
\end{equation}
where $B_n(x)$ is the $n$'th Bernoulli polynomial.
We focus on the following identity due to Biagioli in \cite{biagioli}, involving negative statistics in type $D$.

\begin{theorem}[Biagioli]\label{negDthm}

\[
  \sum_{ k \ge 0 } [k+1]_q^n \, t^k
  = \frac{ \sum_{(\pi,\s) \in D_n } t^{ \dndes(\pi,\s) } q^{ \dnmaj(\pi,\s) } }{ (1-t)(1-tq^n)\prod_{j=1}^{n-1} (1-t^2q^{2j}) } \, .
\]
\end{theorem}

\begin{definition}\label{negDstats} 
Using the order $-1<\cdots<-n<1<\cdots<n$ on $[-n,n]\setminus\{0\}$, for an element $(\pi,\s)\in D_n$, we define $\Des_A(\pi,\s)$, $\nega(\pi,\s)$, and $\des_A(\pi,\s)$ as for the group $B_n$.
Further, we set $\Neg(\pi,\s):=\NNeg(\pi,\s)$.
We define the \emph{type-$D$ negative descent multiset} as
\begin{align*}
\DNDes(\pi,\s) &:= \Des_A(\pi,\s)\cup \{\pi(i) -1 : \s_i=-1 \} \setminus \{0\} \\
&= \Des_A(\pi,\s)\cup \{j-1: j\in \Neg((\pi,\s)^{-1})\setminus \{1\}\} \, .
\end{align*}
The \emph{type-$D$ negative descent statistic} is 
\[
\dndes(\pi,\s):=\#\DNDes(\pi,\s) \, .
\]
The \emph{type-$D$ negative major index} is
\[
\dnmaj(\pi,\s):=\sum_{i\in \DNDes(\pi,\s)}i \, .
\]
\end{definition}

\begin{example}\label{ex:typeD}
Let $(\pi,\s)=[2^1\, 4^1\, 5^0\, 1^1\, 3^1]\in D_5$.
Then $\Des_A(\pi,\s)=\{3\}$ and $\Neg((\pi,\s)^{-1})=\{1,2,3,4\}$, hence $\DNDes(\pi,\s)=\{3\}\cup\{1,2,3\}$ and $\dnmaj(\pi,\s)=9$.
\end{example}

\begin{remark}
Biagioli originally defined $\dnmaj$ and $\dndes$ in \cite{biagioli} using the natural order, but the Theorem~\ref{negDthm} holds for either definition.
\end{remark}

Our multivariate generalization of Theorem~\ref{negDthm} is as follows.
Let $I_{2,n}^*\subseteq I_{2,n}$ denote the elements $(\rho,\s)\in I_{2,n}$ satisfying $\s_1\s_2\cdots \s_n=1$.
It is straightforward from our discussion regarding $I_{r,n}$ that
\[
D_n=\bigcup_{\pi\in S_n}I_{2,n}^*\pi \, .
\]

\begin{theorem}\label{cubeDthm}
\begin{align*}
 \sum_{k \ge 0 } \prod_{j=1}^n [k+1]_{ z_j } \, z_0^k = & \\
\sum_{\pi\in S_n} \sum_{(\rho,\s)\in I_{2,n}^*} & \frac{\displaystyle \prod_{j\in \Des(\pi)} z_0z_{\pi(1)}z_{\pi(2)}\cdots z_{\pi(j)} \prod_{j\in \Neg((\rho,\s)^{-1})\setminus \{1\}} z_0z_{\pi(1)}z_{\pi(2)}\cdots z_{\pi(j-1)} }{\displaystyle (1-z_0)(1-z_0z_{\pi(1)}\cdots z_{\pi(n)})\prod_{j=1}^{n-1} \left(1-z_0^2z_{\pi(1)}^2\cdots z_{\pi(j)}^2 \right) } \, .
\end{align*}
\end{theorem}

\begin{proof}
We begin with the triangulation of $\cone ([0,1]^n)$ into the set of cones 
$
\left\{ \cone(\Delta_\pi) : \, \pi \in S_n \right\}
$
found in the proof of Theorem~\ref{Atheorem}.
For $\cone(\Delta_\pi)$ we use the non-unimodular ray generators 
\[
\e_0, 2(\e_0+\e_{\pi(1)}), \ldots, 2(\e_0+\e_{\pi(1)}+\cdots + \e_{\pi(n-1)}), \e_0+\e_{\pi(1)}+\cdots + \e_{\pi(n)} \, .
\]
There are $2^{n-1}$ integer points in the fundamental parallelepiped for $\cone(\Delta_\pi)$ using these ray generators.
Each such point can be expressed as a linear combination of the middle $n-1$ generators with coefficients $\alpha_j\in \{0,\frac 1 2\}$, plus a sum of shifting vectors for those integer points that need to be shifted away from the boundary of the cone.
As in our proof of Theorem~\ref{wreathcubeBthm}, we will use the technique of shifting the entire parallelepiped.

Thus, every integer point $\p$ in the (shifted) fundamental parallelepiped for $\cone(\Delta_\pi)$ can be uniquely expressed as
\[
\p=\sum_{j\in \Des(\pi)} (\e_0+\e_{\pi(1)}+\cdots + \e_{\pi(j)}) + \sum_{j=1}^{n-1} \alpha_j(\e_0+\e_{\pi(1)}+\cdots + \e_{\pi(j)}) \, 
\]
with $\alpha_j\in \{0, 1\}$.
Associate to the point $\p$ the element $(\rho,\s)\pi \in D_n$, where $\alpha_j=1$ if and only if $j+1\in \Neg((\rho,\s)^{-1})=\Neg([(\rho,\s)\pi]^{-1})$.

As in the proof of Theorem~\ref{wreathcubeBthm}, this correspondence creates a bijection between the elements of $D_n$ and the (appropriately shifted) integer points in the fundamental parallelepipeds for the cones over the $\Delta_\pi$.
Our choice of $(\rho,\s)\pi$ associated to $\p$ is unique because the condition $\alpha_j=1$ if and only if $j+1\in \Neg((\rho,\s)^{-1})=\Neg([(\rho,\s)\pi]^{-1})$ determines the signs placed on the letters $\{2,3,\ldots,n\}$ when $(\rho,\s)\pi$ is written in window notation.
Hence, $\s_1$ is determined from these $n-1$ signs and the fact that $\s_1\cdots\s_n=1$.
This bijection encodes $I_{2,n}^*$ as the integer points in the fundamental parallelepiped for $\cone(\Delta_{\Id})$.

Thus
\[
  \sigma_{\cone(\Delta_\pi)}(z_0,\ldots,z_n) 
  = \sum_{(\rho,\s)\in I_{2,n}^*}  \frac{\displaystyle \prod_{j\in \Des(\pi)} z_0z_{\pi(1)}z_{\pi(2)}\cdots z_{\pi(j)} \prod_{j\in \Neg((\rho,\s)^{-1})\setminus \{1\}} z_0z_{\pi(1)}z_{\pi(2)}\cdots z_{\pi(j-1)} }{\displaystyle (1-z_0)(1-z_0z_{\pi(1)}\cdots z_{\pi(n)})\prod_{j=1}^{n-1} \left(1-z_0^2z_{\pi(1)}^2\cdots z_{\pi(j)}^2 \right) }  \, .
\]
This completes our proof, since from our triangulation it follows that
\[
\sigma_{\cone ([0,1]^n)}(z_0,\ldots,z_n) = \sum_{\pi\in S_n} \sigma_{\cone(\Delta_\pi)}(z_0,\ldots,z_n) \, . \qedhere
\]
\end{proof}

\begin{proof}[Proof of Theorem~\ref{negDthm}]
Setting $t:=z_0$ and $q:=z_1=\cdots =z_n$ in Theorem~\ref{cubeDthm} yields our desired form on the left-hand side of our identity, while the denominator of the right-hand side uniformly becomes 
\[
(1-t)(1-tq^n) \prod_{ j=1 }^{n-1} \left( 1 - t^2 q^{2j} \right) . 
\]
Each element $\displaystyle (\rho,\s)\pi\in \bigcup_{\pi\in S_n} I_{2,n}^*\pi$ contributes to the numerator on the right-hand side of our identity a summand of
\[
\prod_{j\in \Des(\pi)}tq^j \prod_{j\in \Neg([(\rho,\s)\pi ]^{-1})} tq^{j-1} .
\]
Because $\Des(\pi)=\Des_A((\rho,\s)\pi)$, it follows that 
\[
\prod_{j\in \Des(\pi)}tq^j \prod_{j\in \Neg([(\rho,\s)\pi ]^{-1})} tq^{j-1} = t^{\dndes((\rho,\s)\pi)}q^{\dnmaj((\rho,\s)\pi)} ,
\]
hence our proof is complete.
\end{proof}


\bibliographystyle{plain}
\bibliography{eulerianpol}

\end{document}